\documentclass[12pt,a4paper]{amsart}
 \usepackage[titletoc,toc,title]{appendix}
\usepackage{t1enc}
\usepackage[T1]{fontenc}
\usepackage[latin1]{inputenc}
\usepackage{amssymb}
\usepackage{amsthm}
\usepackage{amsmath}
\usepackage{latexsym}
\usepackage{array}
\usepackage{mathrsfs}
\usepackage[all]{xy}
\usepackage{graphics}
\usepackage{xcolor}
\usepackage{bbm}
\usepackage{stmaryrd}
\usepackage{color}
\usepackage{enumerate}
\usepackage{tikz}
\usepackage[hyphens]{url}
\usepackage[pagewise]{lineno}
\usetikzlibrary{matrix}
\title{$J$-invariant of hermitian forms 
over quadratic extensions}
\date{7 August 2018}
\author{Rapha\"{e}l Fino}
\address
{Instituto de Matem\'{a}ticas \\
Ciudad Universitaria\\
UNAM\\
DF 04510, M\'{e}xico}
\address
{{\it Web page:}
{\tt http://www.matem.unam.mx/fino}}
\email {fino {\it at} im.unam.mx}

\thanks{This work has been partially supported by CAPES, Coordination of the Development of the Superior Level Personnel - Brazil, in the framework of the Science Without Borders program with the Federal University of Amazonas (UFAM), Manaus.
}
\numberwithin{equation}{section}
\theoremstyle{definition}
\newtheorem{defi}[equation]{Definition}

\newtheorem{rem}[equation]{Remark}
\newtheorem{lemme}[equation]{Lemma}

\newtheorem{prop}[equation]{Proposition}
\newtheorem{thm}[equation]{Theorem}
\newtheorem{cor}[equation]{Corollary}

\newcommand{\bigslant}[2]{{\raisebox{.2em}{$#1$}\left/\raisebox{-.2em}{$#2$}\right.}}
\begin{document}

\begin{abstract}
We develop the version of the $J$-invariant for hermitian forms over quadratic extensions in a similar way 
Alexander Vishik did it for quadratic forms.
This discrete invariant contains informations about rationality of
algebraic cycles on the maximal unitary grassmannian associated with a hermitian form over a quadratic extension.
The computation of the canonical $2$-dimension
of this grassmannian in terms of the $J$-invariant is provided, as well as a complete motivic decomposition.

\end{abstract}

\maketitle

\tableofcontents

\section{Introduction}

Let $F$ be an arbitrary field and $K/F$ a quadratic separable field extension.
In this article, we define a new discrete invariant $J(h)$ for a non-degenerate $K/F$-hermitian
form $h:V\times V\rightarrow K$.
This invariant is developed on the model of the $J$-invariant for quadratic forms, 
due to Alexander Vishik, see \cite{Grass}, and later generalized to an arbitrary semi-simple algebraic group of inner type by V.\,Petrov, N.\,Semenov and K.\,Zainoulline in \cite{J-inv}.
Let $X$ denote the $F$-variety
of maximal totally $h$-isotropic subspaces of $V$.
The invariant $J(h)$ contains informations about rationality of algebraic cycles on $X$
over a splitting field of $h$.
The same way it was obtained by Nikita A.\,Karpenko and Alexander S.\,Merkurjev for maximal orthogonal grassmannian in the case of quadratic forms (see \cite[Theorem 90.3]{EKM}),
the invariant $J(h)$ notably allows one to recover the canonical $2$-dimension of the maximal 
unitary grassmannian $X$ (Theorem 8.2).

In general, the $J$-invariant has several important applications. For example, {A.\,Vishik} used it in its refutation of the Kaplansky's conjecture on the $u$-invariant of a field (see \cite{uINV}) and 
so did N.\,Semenov when he answered a question by J-P.\,Serre about groups of type $E_8$ (see \cite{mcci}).

In the case of quadratic forms, the Chow motive of the maximal orthogonal grassmannian associated
with a quadratic form
splits as a sum of Tate motives over a splitting field of the quadratic form,
the reason being that there is a nice filtration of the maximal orthogonal grassmannian by
affine bundles.
Because this does not stand in the case of hermitian forms, we use the structure Theorem
\cite[Theorem 15.8]{rcs} by N.\,A.\,Karpenko and the modified Chow ring
$\text{Ch}_K(X):=\text{Ch}(X)/
\;\text{Im}\left(\text{Ch}(X_K)\rightarrow \text{Ch}(X) \right)$, with $\text{Ch}(X)$ the integral
Chow ring $\text{CH}(X)$ modulo $2$.  
These considerations allow one to follow the method introduced by A.\;Vishik for quadratic forms to describe completely the ring
$\text{Ch}_K(X)$ when $h$ is split (equivalently, when $X$ has a rational point) and 
the subring $\text{Im}\left(\text{Ch}_K(X)\rightarrow \text{Ch}_K(X_{F(X)}) \right)$ of rational elements
for arbitrary $h$, where
$F(X)$ is the function field of $X$.
We also work with the category of \textit{$\text{Ch}_K$-motives}, defined from $\text{Ch}_K$, and provide a complete motivic
decomposition of the $\text{Ch}_K$-motive $M^K(X)$ of $X$ in terms of the $J$-invariant $J(h)$ (Theorem 9.4).
The $\text{Ch}_K$-motive $M^K(X)$ is related to the \textit{essential motive} of $X$ (see Remark 9.7).

By a theorem of Jacobson (see \cite[Corollary 9.2]{UG}), the non-degenerate $K/F$-hermitian form $h$
is entirely determined by the associated $F$-quadratic form $q:v\mapsto h(v,v)$, with $V$ considered
as an $F$-vector space. 
Moreover, the $F$-quadratic forms arising this way from 
$K/F$-hermitian forms can be described as the tensor product of a non-degenerate
bilinear form by the norm form of $K/F$, which is an anisotropic binary quadratic from.
Conversely, an $F$-quadratic form defined by such a tensor product is isomorphic to 
the quadratic form arising from the hermitian form induced by the bilinear form
and the quadratic separable field extension $K/F$ given by the discriminant of the binary quadratic form. 
As explained by N.\,A.\,Karpenko in the introduction of \cite{UG}, although these observations 
show that the study of $K/F$-hermitian forms is equivalent to the study of binary divisible quadratic forms
over $F$,
this does not show that the hermitian forms are not worthy of interest.
Indeed, on the one hand, it shows that the class of binary divisible quadratic forms is quite important. On the other hand, it provides the opportunity to use the world of hermitian forms to study such quadratic forms, which can be more appropriate than staying exclusively at the level of quadratic forms,
as illustrated by Proposition 10.1.

The paper is organized as follows.
In section 3, we use the relative cellular space structure on $X$ given by
\cite[Theorem 15.8]{rcs} to get the relation of Proposition 3.4 between Chow rings $\text{Ch}_K$
(defined in section 2) associated with the maximal unitary grassmannian of
a hermitian subform of an isotropic $K/F$-hermitian form $h$. From section 4 to 8, we literally
follow the thread of \cite[\S 86 to \S 90]{EKM}. In this part of the article, we first use the previously
mentioned relation to get a complete description of $\text{Ch}_K(X)$ in the split case
in terms of generators and relations (Theorem 4.9 and Proposition 4.15),
from which we deduce a description of the subring of rational elements in the general case
in terms of those generators (Theorem 5.7). The $J$-invariant $J(h)$ is then defined from the latter description. We also compute some Steenrod operations of cohomological type on $\text{Ch}(X)$
in the split case (Theorem 7.2). In Theorem 8.2, we obtain the canonical 
$2$-dimension of $X$ in terms of $J(h)$, on the model of \cite[Theorem 90.3]{EKM}.
In section 9, using Rost Nilpotence, we provide the complete motivic decompostion
of $M^K(X)$ in terms of $J(h)$ (Theorem 9.4), in the spirit of \cite[Theorem 5.13]{J-inv}.
In the final section 10, we compare the $J$-invariant $J(h)$ of a non-degenerate $K/F$-hermitian
form $h$ with the $J$-invariant $J(q)$ of the associated quadratic form $q$ (Proposition 10.1).

\smallskip
\noindent
{\sc Acknowledgements.} I am grateful to Nikita A.\,Karpenko for initiating me to
the geometric theory of hermitian forms. 
I would like to thank the anonymous referees.

\section{$K$-Chow rings}

Let $F$ be an arbitrary field, $K/F$ a quadratic separable field extension and
$X$ an $F$-variety (i.e., a separated $F$-scheme of finite type).
We denote by $\text{Ch}(X)$ the integral Chow ring $\text{CH}(X)$ modulo $2$.

\medskip

We set
\[\text{Ch}_K(X):=\text{Ch}(X)/
\;\text{Im}\left(\text{Ch}(X_K)\rightarrow \text{Ch}(X) \right),\]
where the homomorphism $\text{Ch}(X_K)\rightarrow \text{Ch}(X)$ is the push-forward
of the projection $X_K \rightarrow X$.

\medskip

Note that $\text{Im}\left(\text{Ch}(X_K)\rightarrow \text{Ch}(X) \right)$ is an ideal by the Projection Formula (\cite[Proposition 56.9]{EKM}), called the \textit{norm ideal}, so that $\text{Ch}_K(X)$ inherits the ring structure of the initial Chow ring.
For example, one has $\text{Ch}_K(\text{Spec}(F))=\mathbb{Z}/ 2\mathbb{Z}$, and for any $F$-variety $X$,  the ring $\text{Ch}_K(X_K)$
is trivial. We write $(\varphi)_K$ for the $K$-Chow groups homomorphism associated with a 
Chow groups homomorphism $\varphi$ which preserves norm ideals.

\medskip

Since the norm ideal is preserved by pull-backs and push-forwards, one can define the additive category of \textit{$\text{Ch}_K$-motives} the same way as the category of Chow motives (see \cite[Chapter XII]{EKM}) but using the Chow rings $\text{Ch}_K$ instead of the usual Chow rings $\text{CH}$.
For a smooth proper $F$-variety $X$, we write $M^K(X)$ for the associated $\text{Ch}_K$-motive.

\begin{rem}
For a field extension $E/F$ and $j\geq 0$, let us denote by  $N_{j}(E)$ the subgroup of the Milnor group $K^M_{j}(E)$ generated by the norms from finite field extensions of $E$ that split the extension $K/F$.
Then the 
cycle module 
$E\mapsto K^M_{\ast}(E)$
over $F$ gives rise to an assignment $E\mapsto K^M_{\ast}(E)/N_{\ast}(E)$.
One can check that the latter is also a cycle module over $F$, in particular,
the fact that residue maps are well-defined comes from the rule \cite[R3b]{Rostchow}. 
Hence, one can consider the cohomology theory associated with this cycle module
(which contained the $K$-Chow groups)
instead of the cohomology theory of the Milnor cycle module
and thus obtain some $\text{Ch}_K$-versions of results for classical Chow groups
(see Propositions 6.5, 8.1 and 9.2).
\end{rem}

\section{Isotropic hermitian forms}

\subsection{Relative cellular spaces}

Let $F$ be a field.

\begin{defi}
Let $X$ be a smooth proper $F$-variety supplied with a filtration $\mathcal{F}$ by closed subvarieties
\[\emptyset = X_{(-1)}\subset X_{(0)} \subset \cdots X_{(n)}=X .\]
The variety $X$ is a \textit{relative cellular space} over a smooth proper 
$F$-variety $Y$ if the associated \emph{adjoint} variety
\[\text{Gr}_\mathcal{F}X=\coprod_{k=0}^{n} X_{(k)}\backslash X_{(k-1)}\]
is a vector bundle over $Y$. The variety $Y$ is called the \textit{base} of $X$.
\end{defi}




For $V$ a finite dimensional $F$-vector space,
we denote by $\Gamma (V)$ the full grassmannian of $F$-subspaces of $V$.
To an epimorphism $p:V\rightarrow V'$ of $F$-vector spaces, one can associate
the filtration
\[\emptyset = \Gamma (V)_{(-1)}\subset \Gamma (V)_{(0)} \subset \cdots \Gamma (V)_{(\text{dim}\, V')}=\Gamma (V)\]
on $\Gamma (V)$ defined as follows:
for any local commutative $F$-algebra $R$ and $0\leq k \leq \text{dim}\, V'$, one has
\[\Gamma (V)_{(k)}(R)=\{N \in \Gamma (V)(R) |\, \Lambda^{k+1}\left(p_R(N) \right)=0 \},\]
where $p_R:V_R\rightarrow V'_R$
is induced by $p$
and $\Lambda^{k+1}$ stands for the $(k+1)$-th exterior power.

\medskip

Let $0\rightarrow V'' \rightarrow V \rightarrow V'\rightarrow 0$ be an exact sequence of $F$-vector
spaces.
The result \cite[Corollary 9.11]{rcs} by 
N.\,A.\,Karpenko asserts that $\Gamma (V)$ supplied with the filtration associated with 
$V\rightarrow V'$ is a relative cellular space over $\Gamma (V'')\times \Gamma (V')$.

\medskip

Moreover, let $K/F$ be a quadratic separable field extension. Suppose that $V$, $V'$ and $V''$ are 
$K$-vector spaces and that the short sequence is an exact sequence of $K$-vector spaces. Then the previous relative cellular structure on $\Gamma (V)$ induces a relative cellular structure on 
the Weil restriction $\Gamma^K(V)$ of the full grassmannian of $K$-subspaces with respect to the extension $K/F$:
$\Gamma^K(V)$ is a relative cellular space over $\Gamma^K (V'')\times \Gamma^K (V')$, see
\cite[Theorem 10.9]{rcs}.
The associated filtration is the restriction of the previous one by $K$-subspaces.

\medskip

Suppose that the $K$-vector space $V$ is isomorphic to a sum of $K$-subspaces $V'\oplus V'' \oplus \tilde{V}$.
Using the exact sequences
\[0\rightarrow V'' \oplus \tilde{V} \rightarrow V \rightarrow V'\rightarrow 0\;\;\; \text{and}\;\;\; 
0\rightarrow V'' \rightarrow V'' \oplus \tilde{V} \rightarrow \tilde{V} \rightarrow 0,\]
and composing the relative structures,
the variety $\Gamma^K(V)$ is turned into a relative cellular space over 
$\Gamma^K (V')\times \Gamma^K (V'')\times \Gamma^K (\tilde{V})$, 
as described in \cite[Example 10.14]{rcs}.

\medskip

Let $h:V\times V\rightarrow K$ be a non-degenerate isotropic $K/F$-hermitian form on $V$.
Denote by $L \subset V$ an isotropic line and set $\tilde{V}=L^{\perp}/L$.
Let $L^{\ast}$ be an arbitrary splitting of $V\rightarrow V/L^{\perp}$.
Applying the observation of the previous paragraph to the decomposition
$V\simeq L \oplus L^{\ast} \oplus \tilde{V}$, one obtains that
$\Gamma^K(V)$ is a relative cellular space over 
$\Gamma^K (L)\times \Gamma^K (L^{\ast})\times \Gamma^K (\tilde{V})$. 

Furthermore, by \cite[Theorem 15.8]{rcs}, the latter relative cellular structure restricts to the $h$-isotropic subspaces
in the following way.
Let $\tilde{h}$ be
the $K/F$-hermitian form on $\tilde{V}$ induced by $h$.
Let $Y$ and $\tilde{Y}$ be the $F$-varieties of totally isotropic subspaces of $h$ and
$\tilde{h}$ respectively. 
Let $Z$ be the Weil transfer of the $K$-variety of $2$-flags of $K$-vector subspaces of $L$ with respect to the extension $K/F$ (note that the 0-dimensional $F$-variety $Z$ is the disjoint union of three copies of 
$\text{Spec}(F)$ and three copies of $\text{Spec}(K)$).
Then $Y$ is a relative cellular space over $Z\times \tilde{Y}$. 
The associated filtration is the restriction of the filtration associated with the relative cellular space structure
of $\Gamma^K(V)$ over $\Gamma^K (L)\times \Gamma^K (L^{\ast})\times \Gamma^K (\tilde{V})$
by $h$-isotropic subspaces.

\medskip

In the same article \cite{rcs}, the author proves that, 
in general, the Chow motive
of a relative cellular space is isomorphic to the Chow motive of its base (\cite[Theorem 6.5]{rcs})
and he describes in \cite[Corollary 6.11]{rcs} how this isomorphism restricts to
the irreducible components of the relative cellular space.

Applied to the previous situation, this gives the following.
Let $X$ be the $F$-variety of maximal totally isotropic subspaces in $h$.
The dimension of such a subspace is $r:=\lfloor \text{dim}(h)/2 \rfloor$.
Note that $X$ is an irreducible component of $Y$.
Besides, the unitary grassmannian $X$ is a projective homogeneous variety
under a projective unitary group of outer type 
(see the introduction of \cite{UG}).
We write $\tilde{X}$ for the 
maximal unitary grassmannian associated with $\tilde{h}$.
Since maximal totally isotropic subspaces
of $\tilde{V}$ are in one-to-one correspondence with those
of $V$ containing $L$, one can view $\tilde{X}$ as a closed subvariety of $X$. 
Let $i:\tilde{X}\hookrightarrow X$ denote the closed embedding.
Let $\beta: \tilde{X} \leadsto X$ be the correspondence given by the scheme of pairs $(W/L,U)$,
where $U$ is a totally isotropic $r$-dimensional $K$-subspace of $V$, $W$ is a totally isotropic 
$r$-dimensional subspace of $L^{\perp}$ containing $L$, and $\text{dim}_K(U+W)\leq r+1$
(correspondences are defined in \cite[\S 62]{EKM}).
Then one has the following Chow motivic decomposition with $\mathbb{Z}/2\mathbb{Z}$-coefficients
\begin{equation}
M(X)\simeq M(\tilde{X})\{d\}\oplus M(\tilde{X})\oplus M,
\end{equation}
where $d=\text{dim}(X)-\text{dim}(\tilde{X})$, the morphism
$M(\tilde{X})\{d\} \rightarrow M(X)$ is given by $\beta$, the morphism $M(\tilde{X})\rightarrow M(X)$ is
given by the class in $\text{Ch}(\tilde{X}\times X )$ of the graph of $i$
and $M$ is a sum of shifts of $M\left(\text{Spec}(K)\right)$.

\medskip

At the level of $K$-Chow groups (introduced in the previous section), decomposition (3.2) implies that
\begin{equation} \text{Ch}_K^{\ast}(X)\simeq \text{Ch}_K^{\ast}(\tilde{X}) \oplus \text{Ch}_K^{\ast -d}(\tilde{X}),\end{equation}
where the injection $\text{Ch}^{\ast}(\tilde{X}) \hookrightarrow \text{Ch}^{\ast}(X)$ coincides with $\beta_{\ast}$ and the injection $\text{Ch}^{\ast -d}(\tilde{X}) \hookrightarrow \text{Ch}^{\ast}(X)$ coincides with $i_{\ast}$.
In particular, if $h$ is split (i.e., if the Witt index $i_0(h)$ of $h$ is equal to $r$), one deduces by induction that  
$\text{Ch}_K(X)$ is a free $\mathbb{Z}/2\mathbb{Z}$-module of rank $2^{r}$.

\medskip

We write $j$ for the open embedding $X\backslash \tilde{X}\hookrightarrow X$ and we set 
\[f:=\beta^t \circ j: X\backslash \tilde{X} \leadsto \tilde{X},\]
with ${\beta}^t$ the transpose of $\beta$.

Since $\text{Im}(i_{\ast})=\text{Ker}(j^{\ast})$ by the localization exact sequence (see \cite[\S 52.D]{EKM}), it follows from
(3.3) that
\[\text{Ch}_K^{\ast}(X)= \text{Im}({(\beta_{\ast})}_K)\oplus \text{Ker}\left((j^{\ast})_K \right).\]
Hence, since $j^{\ast}$ is surjective (see \textit{loc. cit.}), we have obtained the following statement.

\begin{prop}
\textit{The homomorphism} 
\[(f^{\ast})_K :\text{Ch}^{\ast}_K(\tilde{X})\rightarrow \text{Ch}^{\ast}_K(X\backslash \tilde{X}),\]
 \textit{is an isomorphism.}
\end{prop}

The above proposition is crucial for the induction in the proof of Theorem 4.9 in the next section.

\subsection{Associated quadrics}

We use notation introduced in Subsection 3.1.
Let $q:V \rightarrow F$, $v\mapsto h(v,v)$ be the non-degenerate $F$-quadratic form associated with $h$, where $V$ is considered
as an $F$-vector space. Note that $\text{dim}(q)=2\text{dim}(h)$.
We denote by $Q$ the smooth projective quadric of $q$.
Similarly, let $\tilde{q}:\tilde{V} \rightarrow F$ be the non-degenerate $F$-quadratic form associated with the hermitian form $\tilde{h}$ and let us denote by $\tilde{Q}$ the smooth projective quadric of $\tilde{q}$.
Note that since $\tilde{q}$ is also the form induced by $q$ on $P^{\perp}/P$, 
with $P$ the $q$-isotropic $F$-plane corresponding to $L$,
it is Witt-equivalent to $q$.

The \textit{incidence correspondence} $\alpha:\tilde{Q} \leadsto Q$ is given by the scheme of pairs
$(B/P,A)$ of isotropic $F$-lines in $P^{\perp}/P$ and $V$ respectively with $A\subset B$.
By \cite[Lemma 72.3]{EKM}, for $k<i_0(q)$, one has $\alpha_{\ast}(\tilde{l}_{k-2})=l_k$ and $\alpha^{\ast}(l_{k})=\tilde{l}_{k-2}$, where $l_k$ (resp. $\tilde{l}_k$) is the class in $\text{CH}_k(Q)$ 
(resp. $\text{CH}_k(\tilde{Q})$) of a $k$-dimensional totally $q$-isotropic (resp. $\tilde{q}$-isotropic) subspace of $\mathbb{P}_F(V)$ (resp. $\mathbb{P}_F(\tilde{V})$).
If $q$ is split, given an orientation of $Q$ (i.e., the choice of one of the two classes in $\text{CH}(Q)$ of maximal isotropic subspaces), we choose an orientation of $\tilde{Q}$ so that the previous formulas hold for the classes of the respective maximal isotropic subspaces.

\medskip

We write $E$ for the vector bundle of rank $2r$ over $X$ given by the closed subvariety of the trivial bundle
$V\mathbbm{1}=V\times X$ consisting of pairs $(u,U)$ such that $u\in U$ (with $V$ viewed as an $F$-vector space).
Let us denote as $E^{\perp}$ the kernel of the natural morphism $V\mathbbm{1}\rightarrow E^{\vee}$
given by the polar bilinear form associated with the quadratic form $q$, where $E^{\vee}$ is the dual bundle of $E$.
If the dimension of $h$ is even, one has $E^{\perp}=E$. Otherwise, $E$ is a subbundle of 
$E^{\perp}$ of corank $2$.
For our purpose, the vector bundle $E$ is the appropriate hermitian version of the the vector bundle used in \cite[\S 86]{EKM} for the case
of quadratic forms.

The associated projective bundle $\mathbb{P}(E)$ is a closed subvariety of codimension
$2\lfloor (\text{dim}(h)+1)/2\rfloor -1$ of $Q\times X$
and we denote by $in$ the associated closed embedding. 
We write $\gamma$ for the class of $\mathbb{P}(E)$ 
in $\text{CH}(Q\times X)$ and view it as a correspondence $Q \leadsto X$.
Similarly, one can consider the analogous correspondence $\tilde{\gamma}:\tilde{Q} \leadsto \tilde{X}$.

\begin{lemme}
\textit{One has $\gamma\circ \alpha=\beta \circ \tilde{\gamma}$ and 
$\tilde{\gamma}\circ \alpha^t=i^t \circ \gamma$.}
\end{lemme}

\begin{proof}
The proof is almost the same as in the case of quadratic forms, see \cite[Lemma 86.7]{EKM}.
By \cite[Corollary 57.22]{EKM}, it suffices to check the required identities at the level of cycles representing the correspondences.
By definition of the composition of correspondences, the composition $\gamma\circ \alpha$ and $\beta \circ \tilde{\gamma}$ coincide with
the cycle of the subscheme of $\tilde{Q}\times X$ consisting of all pairs $(B/P,U)$
with $\text{dim}_F(B+U)\leq 2r+2$
and the compositions $\tilde{\gamma}\circ \alpha^t$ and 
$i^t \circ \gamma$ coincide with the cycle of the subscheme of $Q\times \tilde{X}$ consisting of all pairs $(A,W/L)$ with $A\subset W$.
\end{proof}

\section{Split maximal unitary grassmannian}

We use notation introduced in Section 3.

In this section, we make the assumption that the non-degenerate $K/F$-hermitian form $h$ is split
and we provide a description of the Chow ring $\text{Ch}_K(X)$ of the maximal unitary grassmannian in terms of generators and relations (Theorem 4.9 and Proposition 4.15). 

The method is the one introduced 
by A.\,Vishik in \cite{Grass} to get the description of the Chow ring modulo $2$ of the maximal orthogonal grassmannian in the split case, except
that we work with the Chow rings $\text{Ch}_K$ and use Proposition 3.4 in replacement of
the filtration by affine bundles on the maximal orthogonal grassmannian. 

The exposition below closely follows the one given in the corresponding part of \cite[\S 86]{EKM} in the case of quadratic forms.
The proofs are very akin to the original ones in \cite[\S 86]{EKM}.

\medskip

Let us write $\text{dim}(h)=2n+2$ or $2n+1$, with $n\geq 0$. 
Note that since $i_0(q)=2i_0(h)$ (see \cite[Lemma 9.1]{UG}), the quadratic form $q$ is split if and only if
$h$ is of even dimension.
In both cases, $\mathbb{P}(E)$ has codimension $2n+1$ in $Q\times X$.

\medskip

If $\text{dim}(h)=2n+2$, the cycle $\gamma$ decomposes as
\begin{equation}
\gamma={l_{2n+1}\times e_0} + {l'_{2n+1}\times e'_0} + {\sum_{k=1}^{2n+1} h^{2n+1-k}\times e_k}
\,\,\,\,\, \text{in}\,\,\, \text{CH}^{2n+1}(Q\times X)
\end{equation}
for some unique elements
$e_k\in \text{CH}^k(X)$, $k\in [0,\;2n+1]$ and $e'_0\in \text{CH}^0(X)$, with $l_{2n+1}$, $l'_{2n+1}$ the two different classes 
of $(2n+1)$-dimensional totally $q$-isotropic subspaces of $\mathbb{P}_F(V)$ and $h^i$ the $i$-th power of the 
pull-back $h^1\in \text{CH}^1(Q)$ of the hyperplane
class $H \in \text{CH}^1(\mathbb{P}_F(V))$ under 
the closed embedding $em : Q\hookrightarrow \mathbb{P}_F(V)$
(see \cite[Propositions 64.3, 68.1 and 68.2]{EKM}).
The variety $X$
is connected because the algebraic unitary group $\text{U}(h)$ is connected and acts transitively on it.
Hence, pulling (4.1) back with respect to the canonical
morphism $Q_{F(X)}\rightarrow Q\times X$, we see that one can choose an orientation of 
$Q$ such that
$e_0=1$ and $e'_0=0$.

\medskip

Otherwise -- if $\text{dim}(h)=2n+1$ -- the cycle $\gamma$ decomposes as
\begin{equation}
\gamma= \gamma'+ {l_{2n-1}\times e_0} + {\sum_{k=2}^{2n+1} h^{2n+1-k}\times e_k}
\,\,\,\,\, \text{in}\,\,\, \text{CH}^{2n+1}(Q\times X)
\end{equation}
for some unique elements $e_k\in \text{CH}^k(X)$, $k\in [2,\;2n+1]$ and $e_0\in \text{CH}^0(X)$,
with $\gamma'$ a cycle such that,
by denoting $p_X$ the projection $Q\times X \rightarrow X$, one has
${p_X}_{\ast}\left((l_{2n+1-k}\times 1)\cdot \gamma'\right)=0$ for any $k\in [2,\;2n+1]$
and ${p_X}_{\ast}\left((h^{2n-1}\times 1)\cdot \gamma'\right)=0$.
Pulling (4.2) back with respect to $Q_{F(X)}\rightarrow Q\times X$, one get that 
$e_0=1$.

\medskip

The multiplication rules in the ring $\text{CH}(Q)$ (see \cite[Proposition 68.1]{EKM}) gives that 
$e_k={p_X}_{\ast}\left((l_{2n+1-k}\times 1)\cdot \gamma\right)$,
for $k\in [2,\;2n+1]$
and also for $k=1$ if $\text{dim}(h)$ is even.
In other words, the correspondence
$\gamma$ satisfies $\gamma_{\ast}(l_{2n+1-k})=e_k$ for $k\in [2,\;2n+1]$
and also for $k=1$  for $h$ of even dimension.
Note also that, in the even-dimensional case, one has $\gamma_{\ast}(h^{2n+1})=1$
and that, in the odd-dimensional case, one has $\gamma_{\ast}(h^{2n-1})=1$.

 \medskip

The relations of the previous paragraph between the correspondence $\gamma$ and the elements $e_k$
can be rewritten as
\begin{equation}
e_k=(p_X\circ in)_{\ast}\circ (p_Q\circ in)^{\ast}(l_{2n+1-k}),
\end{equation}
for $k\in [2,\;2n+1]$ and also for $k=1$ for $h$ of even dimension,
with $p_Q$ the projection $Q\times X \rightarrow Q$ (see \cite[Proposition 62.7]{EKM}).

\begin{lemme}
\textit{One has $e_{2n+1}=[\tilde{X}]$ in} $\text{CH}^{2n+1}(X)$.
\end{lemme}

\begin{proof}
On the one hand, by (4.3) one has $(p_X\circ in)_{\ast}\circ (p_Q\circ in)^{\ast}(l_{0})=e_{2n+1}$.
On the other hand, by denoting $A$ a closed point of $Q$ of degree $1$, the cycle $(p_X\circ in)_{\ast}\circ (p_Q\circ in)^{\ast}(l_{0})$ is 
identified with $[\{U\;|\;A\subset U\}]=[\{U\;|\;A\otimes_F K \subset U\}]$ in $\text{CH}^{2n+1}(X)$.
The latter algebraic cycle is equal to $[\{U\;|\; L \subset U\}]=[\tilde{X}]$.
\end{proof}

\medskip

Let us denote by $\tilde{e}_k \in \text{CH}^k(\tilde{X})$ the elements given by (4.1) or (4.2) for $\tilde{X}$.
Similarly, the correspondence $\tilde{\gamma}$ satisifies
$\tilde{\gamma}_{\ast}(\tilde{l}_{2n-1-k})=\tilde{e}_k$ for
the appropriate integers $k$
with respect to the parity of $\text{dim}(h)$.

The following statement is a direct consequence of Lemma 3.5.

\begin{lemme}
\textit{For all $k\in [0,\;2n-1]\backslash \{1\}$ and also for $k=1$  for $h$ of even dimension, one has}
\begin{enumerate}[(i)]
\item \textit{ $\beta_{\ast}(\tilde{e}_k)=e_k$ in} $\text{CH}(X)$ \textit{;}
\item \textit{  $i^{\ast}(e_k)=\tilde{e}_k$ in} $\text{CH}(\tilde{X})$ \textit{.}
\end{enumerate}
\end{lemme}

\begin{proof}
(i) For $k\in [2,\;2n-1]$ and also for $k=1$  for $h$ of even dimension, one has
\[\beta_{\ast}(\tilde{e}_k)=\beta_{\ast} \circ \tilde{\gamma}_{\ast}(\tilde{l}_{2n-1-k})=
\gamma_{\ast}\circ \alpha_{\ast}(\tilde{l}_{2n-1-k})=\gamma_{\ast}(l_{2n+1-k})=e_k.\]
For $k=0$, one can make the same computation with $\tilde{h}^{2n-1}$ and $h^{2n+1}$
if $h$ is even-dimensional and with $\tilde{h}^{2n-3}$ and $h^{2n-1}$ if
$h$ is odd-dimensional (the incidence correspondence $\alpha$ also satisfies $\alpha_{\ast}(\tilde{h}^i)=h^{i+2}$
for $0\leq i\leq \lfloor \text{dim}(Q)/2\rfloor-2 $, where $\tilde{h}^i$ is the $i$-th power of the hyperplane class $\tilde{h}^1\in
\text{CH}^1(\tilde{Q})$).  


(ii) For $k\in [2,\;2n-1]$ and also for $k=1$  for $h$ of even dimension, one has
\[\begin{array}{rl}
i^{\ast}(e_k)=i^t_{\ast}(e_k)=i^t_{\ast}\circ \gamma_{\ast}(l_{2n+1-k})
=\tilde{\gamma}_{\ast}\circ {\alpha}^t_{\ast}(l_{2n+1-k})= &
\tilde{\gamma}_{\ast}\circ {\alpha}^{\ast}(l_{2n+1-k})\\
=& \tilde{\gamma}_{\ast}(\tilde{l}_{2n-1-k})=\tilde{e}_k.\end{array}\]
\end{proof}

For $I\subset [0,\;2n+1]$ 
or $I\subset [0,\;2n+1]\backslash \{1\}$, depending on whether the dimension of $h$ is respectively even or odd, 
we write $e_I$ for the product of $e_k$ for all $k\in I$.
Similarly, one defines the elements $\tilde{e}_J$ for
$J\subset [0,\;2n-1]$ 
or $J\subset [0,\;2n-1]\backslash \{1\}$.
The following statement is obtained by combining Lemma 4.5(ii) with the Projection Formula and 
{Lemma 4.4}.

\begin{cor}
\textit{One has $i_{\ast}(\tilde{e}_J)=e_J\cdot e_{2n+1}=e_{J\cup\{2n+1\}}$ for every $J\subset [0,\;2n-1]$
or for every $J\subset [0,\;2n-1]\backslash \{1\}$, depending on whether the dimension of $h$ is respectively even or odd.}
\end{cor}


Let us denote by $I_{\text{od}}$ the odd part of a set of integers $I$.

\begin{cor}
\textit{For $h$ of even dimension, the monomial $e_{[1,\;2n+1]_{\text{od}}}=e_1e_3\cdots e_{2n+1}$ is the class of a rational point in} $\text{CH}_0(X)$.
\textit{For $h$ of odd dimension, the monomial $e_{{[3,\;2n+1]}_{\text{od}}}=e_3e_5\cdots e_{2n+1}$ is the class of a rational point in} $\text{CH}_0(X)$.
\end{cor}

\begin{proof}
We induct on $n$. 
If $\text{dim}(h)=1$, one has $X=\text{Spec}(F)$. 
If $\text{dim}(h)=2$, the cycle $e_1$ is the class of a rational point on the curve $X$, see Lemma 4.4.
The conclusion follows from the formulas
$e_{[1,\;2n+1]_{\text{od}}}=i_{\ast}(\tilde{e}_{[1,\;2n-1]_{\text{od}}})$ for $h$ of even dimension and 
$e_{{[3,\;2n+1]}_{\text{od}}}=i_{\ast}(\tilde{e}_{{[3,\;2n-1]}_{\text{od}}})$ for $h$ of odd dimension, see Corollary 4.6.
\end{proof}

\begin{lemme}
\textit{One has $(f^{\ast})_K (\tilde{e}_J)=(j^{\ast})_K(e_J)$ for any $J\subset [0,\;2n-1]$
 or for any 
$J\subset [0,\;2n-1]\backslash \{1\}$, depending on whether the dimension of $h$ is respectively even or odd.}
\end{lemme}

\begin{proof}
By Lemma 4.5(i), it suffices to prove 
that $(f^{\ast})_K$ preserves products to get the conclusion.
Since the composition of correspondences $i^{t}\circ \beta$ coincides with the class of the diagonal in $\text{CH}(\tilde{X}^2)$, one has $i^{\ast}\circ \beta_{\ast}=\text{Id}_{\text{CH}(\tilde{X})}$.
Let $\pi$ denote the quotient map $\text{Ch}_K(X)\rightarrow \text{Ch}_K(X)/ \text{Im}\left((i_{\ast})_K \right)$ and let $\overline{(i^{\ast})_K}$ and $\overline{(j^{\ast})_K}$ be the homomorphisms such that
$\overline{(i^{\ast})_K}\circ \pi=(i^{\ast})_K$ and $\overline{(j^{\ast})_K}\circ \pi=(j^{\ast})_K$
(note that $(i^{\ast})_K\circ (i_{\ast})_K=0$ so $\overline{(i^{\ast})_K}$ is well defined).
One has
\[\overline{(i^{\ast})_K}\circ \left(\overline{(j^{\ast})_K}\right)^{-1}\circ (j^{\ast})_K\circ (\beta_{\ast})_K=
\overline{(i^{\ast})_K}\circ \pi \circ (\beta_{\ast})_K 
=\text{Id}_{\text{Ch}_K(\tilde{X})},\]
that is, the morphism $\overline{(i^{\ast})_K}\circ \left(\overline{(j^{\ast})_K}\right)^{-1}$ is the left inverse of $(f^{\ast})_K$.
Let $x, y\in\text{Ch}_K(\tilde{X})$ and let $z$ be the element of $\text{Ch}_K(\tilde{X})$ such that
\[(f^{\ast})_K(z)=(f^{\ast})_K(x)\cdot (f^{\ast})_K(y).\]
Then, composing the previous identity by $\overline{(i^{\ast})_K}\circ \left(\overline{(j^{\ast})_K}\right)^{-1}$, one get
\[z=\overline{(i^{\ast})_K}\circ \left(\overline{(j^{\ast})_K}\right)^{-1} \circ (j^{\ast})_K \left((\beta_{\ast})_K(x)\cdot (\beta_{\ast})_K(y)  \right),\]
that is
\[z=(i^{\ast})_K\circ (\beta_{\ast})_K(x)\cdot (i^{\ast})_K\circ (\beta_{\ast})_K(y),\]
i.e., 
$z=x\cdot y$.
Therefore, the morphism $(f^{\ast})_K$ preserves products. The lemma is proven.
\end{proof}


Now, one can prove the $K/F$-hermitian analogue of \cite[Proposition 2.4(1)]{Grass} the same way \cite[Theorem 86.12]{EKM}
has been proven for the case of quadratic forms. We still write $e_I$ for the class of $e_I$ in $\text{Ch}_K(X)$.

\begin{thm}
\textit{Let $K/F$ be a quadratic separable field extension and $h$ a split non-degenerate $K/F$-hermitian
form. Let $X$ be the variety of maximal totally isotropic subspaces in $h$.
If $h$ is of even dimension $2n+2$ then the set of all $2^{n+1}$ monomials $e_{I}$ for all subsets $I\subset [0,\; 2n+1]_{od}$ is a basis of
the $\mathbb{Z}/2\mathbb{Z}$-module} $\text{Ch}_K(X)$.
\textit{Otherwise -- if $h$ is of odd dimension $2n+1$ -- the set of all $2^{n}$ monomials $e_{I}$ for all subsets $I\subset [0,\; 2n+1]_{od}\backslash \{1\}$ is a basis of
the $\mathbb{Z}/2\mathbb{Z}$-module} $\text{Ch}_K(X)$.
\end{thm}

\begin{proof}
We use an induction on $n$. If $\text{dim}(h)=2$, the $\mathbb{Z}/2\mathbb{Z}$-module
$\text{Ch}_K^1(X)$ is generated by the class $e_1$ of a rational point (see decomposition (3.3) and Lemma 4.4). 
If $\text{dim}(h)=1$, the statement is obviously true since $X=\text{Spec}(F)$.

Consider the following exact sequence 
\[\xymatrix{0 \ar[r] & \text{Ch}_K(\tilde{X})\ar[r]^{{(i_{\ast})}_K} & \text{Ch}_K(X)\ar[r]^{(j^{\ast})_K} &\text{Ch}_K(X\backslash \tilde{X})\ar[r] &0},\]
associated with the localization exact sequence for classical Chow groups.
In both cases, the induction hypothesis and Corollary 4.6 imply that the set of monomials $e_{I}$
for all $I$ containing $2n+1$ is a basis of the image of ${(i_{\ast})}_K$.
On the other hand, Lemma 4.8, Proposition 3.4 and the induction 
hypothesis imply that, in both cases, the set of all the elements $(j^{\ast})_K(e_{I})$ with $2n+1 \notin I$ 
is a basis of $\text{Ch}_K(X\backslash \tilde{X})$. The conclusion follows.
\end{proof}

Also, one has $e_{2k}=0$ in $\text{Ch}_K(X)$ for all $k\geq 2$ and also for $k=1$ if $\text{dim}(h)$ is even, see (4.13) and Corollary 4.14.


\begin{prop} 
\textit{In} $\text{CH}(X)$, \textit{one has}
\textit{ $c_k(V\mathbbm{1}/E)=2e_k$} 
\textit{for all $k\in [2,\;2n+1]$, $c_{2n+2}(V\mathbbm{1}/E)=0$ and
$c_1(V\mathbbm{1}/E)=2e_1$ if the dimension of $h$ is even.}
\end{prop}

\begin{proof}
We apply $(em \times \text{Id}_X)_{\ast}$ to the cycle $\gamma$.
Assume first that $\text{dim}(h)=2n+2$.
Since ${em}_{\ast}(h^k)=2H^{k+1}$ for all $k\geq 0$,
decomposition (4.1) gives that 
\[[\mathbb{P}(E)]={{em}_{\ast}(l_{2n+1})\times 1} + {\sum_{k=1}^{2n+1} 2H^{2n+2-k}\times e_k}
\,\,\,\,\,\, \text{in}\,\,\,\, \text{CH}^{2n+2}(\mathbb{P}_F(V)\times X).\]

On the other hand, by \cite[Proposition 58.10]{EKM}, one has
\[ [\mathbb{P}(E)]={\sum_{k=0}^{2n+2} H^{2n+2-k}\times c_k(V\mathbbm{1}/E)}.\]
Hence, it follows from the Projective Bundle Theorem (see \cite[Theorem 53.10]{EKM})
that $c_k(V\mathbbm{1}/E)=2e_k$ for all $k\in [1,\;2n+1]$
and $c_{2n+2}(V\mathbbm{1}/E)=0$ in $\text{CH}(X)$.

Assume that $\text{dim}(h)=2n+1$. Decomposition (4.2) gives the identity
\[[\mathbb{P}(E)]=(em \times \text{Id}_X)_{\ast}(\gamma')+{{em}_{\ast}(l_{2n-1})\times 1} + {\sum_{k=1}^{2n+1} 2H^{2n+2-k}\times e_k}\]
in $\text{CH}^{2n+2}(\mathbb{P}_F(V)\times X)$.
Moreover, it follows from the conditions defining the cycle $\gamma'$ in (4.2) 
that $(em \times \text{Id}_X)_{\ast}(\gamma')$ 
has the form $H^{2n+1}\times a$ for some $a\in \text{CH}^{1}(X)$.
Consequently, by using \cite[Proposition 58.10]{EKM}
and the Projective Bundle Theorem as it has been done in the even case, one get
that $c_k(V\mathbbm{1}/E)=2e_k$ for all $k\in [2,\;2n+1]$
and $c_{2n+2}(V\mathbbm{1}/E)=0$ in $\text{CH}(X)$.
\end{proof}

\begin{rem}
The above proposition shows in particular that for \emph{any} non-degenerate $K/F$- hermitian form $h$, the algebraic cycles $2e_k$, for $k\in [2,\;2n+1]$ and also for $k=1$ if $\text{dim}(h)$ is even, are defined over $F$, i.e., these are always rational cycles (the cycles $e_k$ are defined over any splitting field of $h$).
\end{rem}

\begin{rem}
Suppose $\text{dim}(h)=2n+2$. 
Let $Y$ be the maximal orthogonal grassmannian of $q$.
Since a totally $h$-isotropic $K$-subspace is also a totally $q$-isotropic $F$-subspace, 
we have a natural closed embedding $\textit{im}: X\hookrightarrow Y$.
Let us denote by $z_k\in\text{CH}^k(Y)$, $1\leq k \leq 2n+1$, the generators of $\text{CH}(Y)$ 
introduced by A.\,Vishik in \cite{Grass}, 
quadratic analogues of the elements $e_k$ (recall that $q$ is split). 
We write $\mathcal{E}$ for the subbundle of the trivial bundle $V\times Y$ over $Y$ 
given by pairs $(u,U)$ such that $u\in U$.
Note that one has $E=im^{\ast}(\mathcal{E})$.
Hence, since
$z_k$ is half the $k$-th Chern class of the vector bundle $(V\times Y)/\mathcal{E}$
(see \cite[Proposition 86.13]{EKM} or \cite[Proposition 2.1]{uINV}) and $\text{CH}(X)$ is torsion-free 
(since the $K/F$- hermitian form $h$ is split, see \cite{koc} for example),
it follows from Proposition 4.10 that 
\[e_k=im^{\ast}(z_k)\]
in $\text{CH}(X)$ for all $1\leq k \leq 2n+1$
(this can also be obtained by comparing the decompositions (4.1) and \cite[(86.4)]{EKM}).
In the odd-dimensional case, the variety $X$ is naturally a closed subvariety of the second to last grassmannian of $q$.
\end{rem}

We set $e_k=0$ for $k>2n+1$ (which makes sense since by Proposition 4.10 the vector bundle
$V\mathbbm{1}/E$ is in both cases of rank $2n+2$ with trivial top Chern class).
By duality, one has $V\mathbbm{1}/E^{\perp}\simeq E^{\vee}$.
Hence, in the Grothendieck ring $K(X)$, one has 
\[ [E^{\perp}/E]=[E^{\perp}]-[E]=-[E^{\vee}]-[E]=[\left( V\mathbbm{1}/E\right)^{\vee}]+[V\mathbbm{1}/E] \]
modulo the trivial bundles.
Therefore, by the Whitney Sum Formula (\cite[Proposition 54.7]{EKM}), one has 
$c(E^{\perp}/E)=c \left( \left(V\mathbbm{1}/E\right)^{\vee} \right)\circ c(V\mathbbm{1}/E)$.
Thus,
since $\text{CH}(X)$ is torsion-free, it follows from
Proposition 4.10 that, for even-dimensional $h$, one has
\[e_k^2-2e_{k-1} e_{k+1}+2e_{k-2}e_{k+2}-\cdots+(-1)^{k-1}2e_1e_{2k-1}+(-1)^ke_{2k}=0\]
in $\text{CH}(X)$, for all $k\geq 1$, and for odd-dimensional $h$, one has
\[e_k^2-2e_{k-1}e_{k+1}+2e_{k-2}e_{k+2}-\cdots+(-1)^{k-1}c_1(V\mathbbm{1}/E)\cdot e_{2k-1}+(-1)^ke_{2k}=0\]
in $\text{CH}(X)$, for all $k\geq 2$
(recall that $E^{\bot}=E$ if $\text{dim}(h)$ is even and $E^{\bot}/E$ has rank $2$ otherwise).
Moreover, by using decomposition (3.3) and an induction on $n$, one obtains that,
for odd-dimensional $h$, the group $\text{Ch}_K^1(X)$ is trivial.
Consequently, in both even and odd situations, one has 
\begin{equation}e_k^2=e_{2k}\end{equation}
in $\text{Ch}_K(X)$, for all $k\geq 2$, and also for $k=1$ if $\text{dim}(h)$ is even.
The relations (4.13) are also a consequence of the respective result for the $z_k$'s, c.f., Remark 4.12 and \cite[(86.15)]{EKM}.

\begin{cor}
\textit{One has $e^2_{k}=0$ in} $\text{Ch}_K(X)$ \textit{for all $k\geq 2$, and also for $k=1$ if $\text{dim}(h)$ is even.}
\end{cor}

\begin{proof}
We proceed by induction on $n$. For $n=0$, the conclusion is true by dimensional reasons.
By (4.13), the conclusion is true for $k>n$. Let $1\leq k \leq n$ (or $2\leq k \leq n$ if $\text{dim}(h)$ is odd).
By codimensional reasons, it suffices to prove that $(j^{\ast})_K(e^2_{k})=0$ in 
$\text{Ch}^{2k}_K(X\backslash \tilde{X})$  to get that
$e^2_{k}=0$ in $\text{Ch}^{2k}_K(X)$ (c.f., the exact sequence in the proof of Theorem 4.9).
Moreover, one has $(j^{\ast})_K(e^2_{k})=(f^{\ast})_K({\tilde{e}_{k}}^2)$ (see Lemma 4.8 and its proof).
Since ${\tilde{e}_{k}}^2=0$ in $\text{Ch}_K(\tilde{X})$ by the induction hypothesis, the proof is complete. 
\end{proof}

We have obtained that $e_{2k}=0$ in $\text{Ch}_K(X)$ for all $k\geq 2$ and also for $k=1$ if $\text{dim}(h)$ is even.
We are now able to determine the
 multiplicative structure for the ring $\text{Ch}_K(X)$ the same way it has been done for 
the case of quadratic forms.

\begin{prop}
\textit{There is a ring isomorphism}
\[\text{Ch}_K(X) \simeq \bigotimes_{m\leq k\leq n} \left(\bigslant{\mathbb{Z}/2\mathbb{Z}[e_{2k+1}]}{(e_{2k+1}^2)} \right),\]
\textit{with $m=0$ if the dimension of $h$ is even and $m=1$ otherwise.}
\end{prop}

\begin{proof}
Assume first that $h$ is of even dimension. Let us denote by $\mathcal{R}$ the factor ring of the polynomial ring
$\mathbb{Z}/2\mathbb{Z}[t_0,t_1,\dots, t_n]$ by the ideal generated by the monomials
$t_k^2$, for $0\leq k\leq n$.
Then, by Corollary 4.14, the ring homomorphism 
\[\begin{array}{rrcl}
\varphi : & \mathcal{R} & \rightarrow & \text{Ch}_K(X) \\
 & t_k & \mapsto & e_{2k+1}
\end{array}\]
is well defined. Furthermore, it follows from Theorem 4.9 that $\varphi$ is
surjective. Since the classes of the monomials $t_0^{r_0}t_1^{r_1}\cdots t_n^{r_n}$, with $r_k=0$ or $1$ for every $k$, generate $\mathcal{R}$, the ring homomorphism $\varphi$ is also injective by Theorem 4.9.
One proceeds the same way for $h$ of odd dimension, except there is no variable $t_0$.
\end{proof}

\section{General maximal unitary grassmannian}

In this section, we use notation introduced in the previous sections.
We do not make any assumption on the isotropy of the 
non-degenerate $K/F$-hermitian form $h$.

For any scheme $Y$ over $F$, we write $\bar{Y}$ for $Y\times \text{Spec}\,F(X)$,
where $F(X)$ is the function field of the maximal unitary grassmannian $X$ (note that $K\otimes_F F(X)$ is still a field).
Following the path set by A.\,Vishik in \cite{Grass}, we describe the subring, 
\[\overline{\text{Ch}}_K(X):=
\text{Im}\left(\text{Ch}_K(X)\rightarrow \text{Ch}_K(\bar{X})    \right).\]
of \textit{rational} elements, c.f. Theorem 5.7 
(actually, this description does not depend on the choice of a splitting field of $h$ which does not split $K$).
We use similar notation and vocabulary for classical Chow rings and certain products
of $F$-varieties.

The exposition in this section follows the thread of \cite[\S 87]{EKM}. The proofs
are sometimes very similar (the proof of Theorem 5.7 compared to the original
\cite[Theorem 87.7]{EKM} for example).

\medskip

We consider the elements $e_k\in \text{CH}^k(\bar{X})$ introduced in the previous section.

\begin{prop}
\textit{-- If} $\text{dim}(h)=2n+2$
\textit{then the elements $(e_k\times 1)+(1\times e_k)$ in} 
$\text{CH}(\bar{X}^2)$ \textit{are rational for all $k\in [1,\;2n+1]$;} 

\textit{-- If} $\text{dim}(h)=2n+1$
\textit{then the elements $(e_k\times 1)+(1\times e_k)$ in} 
$\text{Ch}_K(\bar{X}^2)$ \textit{are rational for all $k\in [2,\;2n+1]$.} 

\end{prop}

\begin{proof}
Suppose first that $\text{dim}(h)=2n+2$.
Here we use notation and material introduced in Remark 4.12. 
Let $k\in [1,\;2n+1]$. Since the element
$(z_k\times 1)+(1\times z_k) \in
\text{CH}(\bar{Y}^2)$ is rational (see \cite[Proposition 86.17 and Corollary 87.3]{EKM}) and one has $im^{\ast}(z_k)=e_k$, the element
$(e_k\times 1)+(1\times e_k) \in
\text{CH}(\bar{X}^2)$ is also rational.

Suppose now that $\text{dim}(h)=2n+1$.
In that case, let us denote by $im$ the natural closed embedding $X\hookrightarrow Y_{2n}$ of the variety $X$ into the second to last grassmannian of $q$.
Let us denote by $z^{-1}_k\in \text{CH}^k\left(\bar{Y}_{2n}\right)$, $1\leq k \leq 2n+1$, the $Z$-type generators of $\text{CH}^k\left(\bar{Y}_{2n}\right)$ (see \cite[\S 2]{uINV}).
Note that it follows from \cite[Proposition 2.1]{uINV} and Proposition 4.10
that, for any $2\leq k \leq 2n+1$, one has 
\[im^{\ast}(z^{-1}_k)=e_k\]
in $\text{CH}(\bar{X})$. 
Moreover, we write $Y$ for the maximal grassmannian of $q$ and
$z_k\in\text{CH}^k(Y)$, $1\leq k \leq 2n$, for the generators of $\text{CH}(\bar{Y})$ 
introduced by A.\,Vishik in \cite{Grass}.
We consider the diagram
\[\xymatrix{
Y_{2n} & Y_{2n,2n+1} \ar[l]_{f} \ar[r]^{g} & Y,
}\]
where $Y_{2n,2n+1}$ is the partial orthogonal flag variety of totally $q$-isotropic $F$-vector subspaces of dimension $2n$ and $2n+1$, and the morphisms $f$ and $g$ are the natural projections.
By \cite[Lemma 2.1]{uINV}, there is an element $c_1\in \text{CH}^1(Y_{2n,2n+1})$ such that, for any $2\leq k \leq 2n+1$, one has
\[f^{\ast}(z^{-1}_k)= c_1\cdot g^{\ast}(z_{k-1})+g^{\ast}(z_{k})\]
(with $z_{2n+1}=0$).
Let $e\in \text{CH}(Y_{2n,2n+1})$ be the class of a generic point. Using the previous identity and the Projection Formula, one get that
\begin{equation} z^{-1}_k=f_{\ast}\left(e\cdot c_1\cdot g^{\ast}(z_{k-1}) \right)+f_{\ast}\left( e\cdot g^{\ast}(z_k) \right) \end{equation}
for any $2\leq k \leq 2n+1$.
Since the element $(z_k\times 1)+(1\times z_k) \in
\text{CH}(\bar{Y}^2)$ is rational for any $1\leq k \leq 2n+1$, the element
\begin{multline*} (f_{\ast})^{\times 2}\left((e\times e)\cdot (c_1\times 1+1\times c_1)\cdot (g^{\ast})^{\times 2}\left((z_{k-1}\times 1)+(1\times z_{k-1})\right)\right) \\ +
\left(f_{\ast}\right)^{\times 2}\left((e\times e)\cdot (g^{\ast})^{\times 2}\left( (z_k\times 1)+(1\times z_k)\right) \right)\end{multline*}
is rational for any $2\leq k \leq 2n+1$.
Furthermore, by (5.2), the latter cycle can be rewritten as
\begin{equation}(z^{-1}_k\times 1)+(1\times z^{-1}_k)+f_{\ast}\left(e\cdot  g^{\ast}(z_{k-1}) \right) \times  f_{\ast} (e\cdot c_1) +
f_{\ast}(e\cdot c_1)\times f_{\ast}\left( e\cdot g^{\ast}(z_{k-1}) \right).\end{equation}
Since $f_{\ast} (e\cdot c_1)\in \text{CH}^1(Y_{2n})$ and $\text{Ch}_K^1(\bar{X})$ is trivial (see paragraph right before
Corollary 4.14), the conclusion is obtained by taking the image of the rational cycle (5.3) under $im^{\ast}$ and then projecting to $\text{Ch}_K(\bar{X})$.
\end{proof}

For every subset $I\subset [1,\; 2n+1]_{\text{od}}$, 
with $1\notin I$ if $\text{dim}(h)=2n+1$, 
we set
\begin{equation}
x_I:=\prod_{k\in I}\left((e_k\times 1)+(1\times e_k) \right)\in\overline{\text{Ch}}_K(X^2).
\end{equation}

\begin{lemme}
\textit{For any subsets} $I,J\subset [1,\; 2n+1]_{\text{od}},$
\textit{with $1\notin I\cup J$ if} $\text{dim}(h)=2n+1,$
\textit{one has in} $\text{Ch}_K(\bar{X})$
\[(x_J)_{\ast}(e_I)=\left \{\begin{array}{ll} e_{I\cap J} & \;\;\textit{if}\;\; I\cup J=[1,\; 2n+1]_{\text{od}}\;,\;\; \textit{or} \;\; I\cup J=[3,\; 2n+1]_{\text{od}} \\
&\;\;\textit{for odd-dimensional}\;\; h ; \\
 0 &\;\; \textit{otherwise} .\end{array}\right.\]
\end{lemme}

\begin{proof}
Assume that $\text{dim}(h)=2n+2$.
Since
\[x_J=\sum_{J_1\subset J} e_{J_1}\times e_{J \backslash J_1} ,\]
one has 
\[(x_J)_{\ast}(e_I)=\sum_{J_1\subset J}  \text{deg}(e_I \cdot e_{J_1})e_{J \backslash J_1}\,\,,\]
in $\text{Ch}_K(\bar{X})$, where $\text{deg}:\text{Ch}_K(\bar{X})\rightarrow 
\text{Ch}_K(\text{Spec}\,F(X))=\mathbb{Z}/2\mathbb{Z}$ is 
the degree homomorphism associated with the push-forward of the structure morphism.
Therefore, it suffices to show that for any subsets $I,J_1\subset [1,\; 2n+1]_{\text{od}}$, one has
\begin{equation} \text{deg}(e_I \cdot e_{J_1})=\left \{\begin{array}{ll} 1\;\;\text{mod}\;2 & \;\;\text{if}\;\; J_1=[1,\; 2n+1]_{\text{od}} \backslash I  ; \\
 0\;\;\text{mod}\;2 &\;\; \text{otherwise} ,\end{array}\right.\end{equation}
to get the conclusion. For $J_1=[1,\; 2n+1]_{\text{od}} \backslash I$, 
this follows from Corollary 4.7.
For $J_1\neq [1,\; 2n+1]_{\text{od}} \backslash I$,
using Corollary 4.14 (or Theorem 4.9), one get that $e_I\cdot e_{J_1}$
is either zero or the monomial $e_L$ for some $L$ different from $[1,\; 2n+1]_{\text{od}}$,
thus $\text{deg}(e_I \cdot e_{J_1})=0\;\;\text{mod}\;2$.
The proof in the odd case is similar. 
\end{proof}

We are now able to prove the $K/F$-hermitian analogue of \cite[Main Theorem 5.8]{Grass}. 

\begin{thm}
\textit{Let $K/F$ be a quadratic separable field extension and $h$ a non-degenerate $K/F$-hermitian
form of dimension $2n+2$ or $2n+1$. Let $X$ be the variety of maximal totally isotropic subspaces in $h$.
Then the ring} $\overline{\text{Ch}}_K(X)$ \textit{is generated by all} $e_k$, $k\in [3,\; 2n+1]_{\text{od}}$
\textit{and also $k=1$ for $h$ of even dimension, such that} $e_k\in \overline{\text{Ch}}_K(X)$.
\end{thm}

\begin{proof}
Assume that $\text{dim}(h)=2n+2$. By Theorem 4.9, one has to show that if an element
 $\alpha=\sum a_Ie_I \in  \text{Ch}_K(\bar{X})$ 
(with $I\subset [1,\; 2n+1]_{\text{od}}$ and $a_I\in \mathbb{Z}/2\mathbb{Z}$) 
is rational then
for every $I$ satisfying $a_I=1$ and any $k\in I$, the element
$e_k\in \text{Ch}_K(\bar{X})$ is rational. 

One may assume that $\alpha$ is homogeneous. We induct on the number of
nonzero coefficients of $\alpha$.

Let $I$ be a subset with largest $|I|$ such that $a_I=1$. Let $k\in I$ and set 
\[J=\left([1,\; 2n+1]_{\text{od}}\backslash I\right) \cup \{k\}.\] 
We claim that $(x_J)_{\ast}(\alpha)=e_k$.
By Lemma 5.5, it suffices to prove that if $I'$ is a subset such that $a_{I'}=1$
and $I'\cup J=[1,\; 2n+1]_{\text{od}}$ then $I'=I$.
Since $I'\cup J=[1,\; 2n+1]_{\text{od}}$, one has 
\[I'=\left([1,\; 2n+1]_{\text{od}}\backslash J\right) \cup (J\cap I').\]
By maximality of $|I|$, the subset $J\cap I'$ is either empty or consists of a single element.
Hence, it follows from the homogeneity of $\alpha$ that $J\cap I'=\{k\}$, that is $I'=I$.
The claim is proven.

Thus, since $x_J$ is rational, the cycle $e_k$ is also rational, for all $k\in I$. Consequently, the elements
$e_I$ and $\alpha - e_I$ are rational.
By the induction hypothesis, every element $e_k$ appearing in the decomposition
of $\alpha-e_I$ is rational and it is therefore so for $\alpha$.
This concludes the even case and the odd case can be treated similarly. 
\end{proof}

\section{The invariant $J(h)$}

In this section, we define a new invariant of non-degenerate $K/F$-hermitian forms on the model of the
$J$-invariant for non-degenerate quadratic forms defined by A.\,Vishik in \cite{Grass}
(althoug this section follows the thread of \cite[\S 88]{EKM}, where the latter is defined in the "opposite way").

\medskip

Let $h$ be a non-degenerate $K/F$-hermitian form of dimension $2n+2$ or $2n+1$ and 
$X$ the variety of maximal totally isotropic subspaces in $h$.
We use notation introduced in the previous sections.
We still denote by $e_k$ the generators of $\text{Ch}_K(\bar{X})$.
The discrete \textit{$J$-invariant} $J(h)$ is defined as follows:

\[J(h)=\left \{\begin{array}{ll}\left\{k\in [1,\; 2n+1]_{\text{od}}\;\;\; \text{with}\;\;\; e_k\in \overline{\text{Ch}}_K(X)\right\} & \;\;\text{if}\;\;\text{dim}(h)=2n+2 ; \\
\left\{k\in [3,\; 2n+1]_{\text{od}}\;\;\; \text{with}\;\;\; e_k\in \overline{\text{Ch}}_K(X)\right\} & 
\;\;\text{if}\;\;\text{dim}(h)=2n+1 .\end{array}\right.\]

\medskip


For a subset $I\subset [1,\; 2n+1]$ let us denote by $||I||$ the sum of all $k\in I$.

\begin{prop}
\textit{The biggest codimension $i$ such that} 
$\overline{\text{Ch}}^i_K(X)\neq 0$ \textit{is equal to} ${||J(h)||}$.
\end{prop}

\begin{proof}
The element $\prod_{k\in J(h)} e_k \in {\overline{\text{Ch}}_K}(X)$ is non-trivial by Theorem 4.9 and has the biggest codimension amongst the non-trivial elements of ${\overline{\text{Ch}}_K}(X)$ by Theorem 5.7.
\end{proof}

\begin{prop}
\textit{A non-degenerate $K/F$-hermitian form $h$ is split if and only if $J(h)$ is maximal.}
\end{prop}

\begin{proof}
If $h$ is split then the fact that $J(h)$ is maximal follows from the definition.
If $J(h)$ is maximal then, by Corollary 4.7,
the class of a rational point of $\bar{X}$ belongs to ${\overline{\text{Ch}}_K}(X)$.
Consequently, the variety $X$ admits a closed point $x$ of odd degree 
(recall that the degree map is well defined on $\text{Ch}_K$). Combining the fact that the 
residue field $F(x)$ is a splitting field of $h$ with Springer's Theorem for quadrics, one get the identities
\[\lfloor \text{dim}(h)/2  \rfloor = i_0(h_{F(x)})=i_0(q_{F(x)})/2=i_0(q)/2=i_0(h).\]
Therefore $h$ is split.
\end{proof}

\begin{lemme}
\textit{Let $h=\tilde{h}\,\bot \,\mathbb{H}$. Then $J(h)=J(\tilde{h})\cup\{2n+1\}$.}
\end{lemme}

\begin{proof}
Since $e_{2n+1}=[\tilde{X}]$ (see Lemma 4.4), one has $2n+1\in J(h)$.
Let $i\leq 2n-1$.
From decomposition (3.3) (where $d=2n+1$), one get $\text{Ch}_K^i(X)\simeq \text{Ch}_K^i(\tilde{X})$ with
$e_i$ corresponding to $\tilde{e}_i$ by Lemma 4.5(i). 
The conclusion follows.
\end{proof}

\begin{cor}
\textit{Let $h$ and $h'$ be Witt-equivalent $K/F$-hermitian forms with $h\simeq h'\,\bot \,j\mathbb{H}$. 
Then $J(h)=J(h')\cup \{2n+1, 2n-1,\dots, 2n+1-2(j-1)\}$}.
\end{cor}

The following statement is the $\text{Ch}_K$-version of
the result \cite[Lemma 88.5]{EKM} of N.\,A.\,Karpenko 
and A.\,S.\, Merkurjev for classical Chow groups (see Remark 2.1).

\begin{prop}
\textit{Let $Z$ be a smooth $F$-variety and $Y$ an equidimensional $F$-variety. Given an integer $m$ such that for any nonnegative integer $i$ and any point $y\in Y$ of codimension $i$
the change of field homomorphism }
\[\text{Ch}_K^{m-i}(Z)\longrightarrow \text{Ch}_K^{m-i}(Z_{F(y)})\]
\textit{is surjective, the change of field homomorphism}
\[\text{Ch}_K^m(Y)\longrightarrow \text{Ch}_K^m(Y_{F(Z)})\]
\textit{is also surjective.}
\end{prop}


For any integer $k$, let us denote by $X_k$ the $F$-variety of $k$-dimensional totally isotropic subspaces in $h$
(so $X_{\lfloor \text{dim}(h)/2 \rfloor}=X$ and $X_k=\emptyset$ for $k\notin [0,\; \lfloor \text{dim}(h)/2 \rfloor]$).
By the result \cite[Corollary 7.3]{UG} (which is also a consequence of \cite[Theorem 15.8]{rcs}),
for any integer $k$,
there is a Chow motivic decomposition with $\mathbb{Z}/2\mathbb{Z}$-coefficients
\begin{equation}
M(X_k)\simeq M_k \oplus M,
\end{equation}
where $M$ is a sum of shifts of $M\left(\text{Spec}(K)\right)$ and $M_k$ splits as a sum of Tate motives over any splitting field of $h$.
The motive $M_k$ is defined by $X_k$ uniquely up to an isomorphism and is called the \emph{essential}
motive of $X_k$.
 
\begin{lemme}
\textit{The change of field homomorphism} $\text{Ch}_K^i(X)\rightarrow \text{Ch}_K^i(X_{F(X_1)})$ \textit{is
surjective for any $i\leq 2n$ if} $\text{dim}(h)=2n+2$ \textit{and for any $i\leq 2n-1$ if} $\text{dim}(h)=2n+1$.
\end{lemme}

\begin{proof}
By Proposition 6.5, it is sufficient to prove that for any $x \in X$
the change of field homomorphism $\text{Ch}_K^i(X_1)\rightarrow \text{Ch}_K^i({X_1}_{\,F(x)})$ is
surjective, for any $i\leq 2n$ if $\text{dim}(h)=2n+2$ and for any $i\leq 2n-1$ if $\text{dim}(h)=2n+1$, to get the conclusion.

It follows from decomposition (6.6) 
that 
\begin{equation}
\text{Ch}_K(X_1)\simeq \text{Ch}_K(M_1).
\end{equation} 
Furthermore, by \cite[Corollary 9.6]{UG}, one has the following Chow motivic decomposition with $\mathbb{Z}/2\mathbb{Z}$-coefficients
\begin{equation}
M(Q)\simeq M_1 \oplus M_1\{1\}
\end{equation}
(where $Q$ is the smooth projective quadric associated with the non-degenerate quadratic form $q:V \rightarrow F$, $v\mapsto h(v,v)$).

Combining (6.8) with (6.9), we see that it suffices to show that for any $x \in X$ the change of field homomorphism 
\begin{equation} \text{Ch}_K^i(Q)\rightarrow \text{Ch}_K^i({Q}_{\,F(x)}) \end{equation}
is surjective, for any $i\leq 2n$ if $\text{dim}(h)=2n+2$ and for any $i\leq 2n-1$ if $\text{dim}(h)=2n+1$, to get the conclusion.
In fact, (6.10) is already surjective at the level of integral Chow groups. 
Indeed, since $F(x)$ is a splitting field
of the hermitan form $h$, one has $i_0(q_{\,F(x)})=2n+2$ or $2n$ depending on whether $\text{dim}(h)$
is respectively even or odd. Therefore, the group $\text{CH}^i(Q_{F(x)})$ is generated by $h^i$ (always rational) for $i\leq 2n$ or $i\leq 2n-1$ depending on whether $\text{dim}(h)$
is respectively even or odd (see \cite[\S 1]{ARC2} for example).

This completes the proof.
\end{proof}

\begin{cor}
$J(h)\subset J(h_{\,F(X_1)})\subset J(h) \cup \{2n+1\}.$
\end{cor}

The following proposition relates the set $J(h)$ and the absolute Witt indices of $h$.
It follows from Corollaries 6.4 and 6.11.

\begin{prop}
\textit{Let $h$ be a non-degenerate $K/F$-hermitian form of dimension $2n+2$ or $2n+1$ with height 
$\mathfrak{h}(h)$. Then
$J(h)$ contains the complementary of the set}
\[\{2n+1-2j_0(h), 2n+1-2j_1(h),\dots , 2n+1- 2j_{\mathfrak{h}(h)-1}(h)\}\]
\textit{in $[1,\;2n+1]_{\text{od}}$, excluding $1$ for $h$ of odd dimension.
In particular, $|J(h)|\geq n- \mathfrak{h}(h)$ and the inequality is strict for $h$ of even dimension.}
\end{prop}

\section{Steenrod operations}

Let $h$ be a non-degenerate $K/F$-hermitian form on $V$ of dimension $2n+2$ or $2n+1$
and let $X$ be the variety of maximal totally isotropic subspaces in $h$.
This section is the hermitian replica of \cite[\S 89]{EKM}, where we compute the Steenrod operations
on $\text{Ch}(\bar{X})$.

\medskip

We use notation introduced in the previous sections and we write $\pi_X$ and $\pi_Q$ for 
the respective compositions $p_X\circ in$ and $p_Q\circ in$.
Let $\mathcal{L}$ be the canonical line bundle over $\mathbb{P}(E)$ and $\mathcal{T}$ the relative tangent
bundle of $\pi_X$.
By \cite[Example 104.20]{EKM}, there is an exact sequence of vector bundles over $\mathbb{P}(E)$:
\[0\rightarrow \mathbbm{1}\rightarrow \mathcal{L} \otimes \pi_X^{\ast}(E)\rightarrow \mathcal{T} \rightarrow 0.\]
Hence, since $c_i(E)=0$ in $\text{Ch}(\bar{X})$ for all $i>1$, $c_1(E)=0$ in $\text{Ch}(\bar{X})$ for even-dimensional $h$ and $c_1(E)=0$ in $\text{Ch}_K(\bar{X})$ for odd-dimensional $h$ (follows from Proposition 4.10
and the fact that
$\text{Ch}_K^1(\bar{X})$ is trivial in the odd case), one deduces from the Whitney Sum Formula and \cite[Remark 3.2.3 (b)]{Ful} that 
\[ c(\mathcal{T})=c\left(\mathcal{L} \otimes \pi_X^{\ast}(E)\right)=c\left(\mathcal{L} \otimes \mathbbm{1}^{2r}\right)=c(\mathcal{L})^{2r}\]
in $\text{Ch}(\bar{X})$ if $\text{dim}(h)$ is even, in $\text{Ch}_K(\bar{X})$ otherwise,
with  $r=\lfloor \text{dim}(h)/2 \rfloor$ (recall that $E$ has rank $2r$).
Furthermore, since $\mathcal{L}$ coincides with the pull-back with respect to $\pi_Q$ of the canonical line bundle over $Q$, 
one has $c(\mathcal{L})=1+\pi_Q^{\ast}(h^1)$ in $\text{CH}(Q)$. Consequently, one has
\begin{equation} c(\mathcal{T})=\left(1+\pi_Q^{\ast}(h^1)\right)^{2r} \end{equation}
in $\text{Ch}(\bar{X})$ if $\text{dim}(h)$ is even, in $\text{Ch}_K(\bar{X})$ otherwise.

\medskip

The following statement is the $K/F$-hermitian analogue of the result 
 \cite[Theorem 4.1]{Grass} 
for quafratic forms.

\begin{thm}
\textit{Assume} $\text{char}(F)\neq 2$. \textit{Let $h$ be a non-degenerate $K/F$-hermitian form of dimension $2n+2$ or $2n+1$ and $X$ the variety of maximal totally isotropic subspaces in $h$.
Let} $S_{\bar{X}}: \text{Ch}(\bar{X})\rightarrow \text{Ch}(\bar{X})$ \textit{denote the Steenrod operation of cohomological
type on $\bar{X}$. Then one has}
\[S_{\bar{X}}^i(e_k)={k \choose i} e_{k+i}\]
\textit{in} $\text{Ch}(\bar{X})$ \textit{if} $\text{dim}(h)$ \textit{is even, in} $\text{Ch}_K(\bar{X})$ \textit{otherwise, for all $i$ and $k\in [2,\; 2n+1]$
and also for $i=k=1$ if} $\text{dim}(h)$ \textit{is even.}
\end{thm}

\begin{proof}
By \cite[Corollary 78.5]{EKM}, one has $S_Q(l_{2n+1-k})=(1+h^1)^{2r+k}\cdot l_{2n+1-k}$.
It follows from (4.3), (7.1) and \cite[Theorem 61.9 and Proposition 61.10]{EKM} that one has
\[\begin{array}{rl}
S_{\bar{X}}(e_k) & =S_X\circ {\pi_X}_{\ast} \circ \pi_Q^{\ast}(l_{2n+1-k}) \\
 & ={\pi_X}_{\ast} \circ c(-\mathcal{T})\circ S_{\mathbb{P}(E)} \circ  \pi_Q^{\ast}(l_{2n+1-k}) \\
 & ={\pi_X}_{\ast}\left( \left(1+\pi_Q^{\ast}(h^1)\right)^{-2r} \cdot \pi_Q^{\ast}\circ S_Q(l_{2n+1-k}) \right) \\
 & = {\pi_X}_{\ast} \circ \pi_Q^{\ast} \left( (1+h^1)^{-2r} \cdot (1+h^1)^{2r+k} \cdot l_{2n+1-k} \right) \\
 & = {\pi_X}_{\ast} \circ \pi_Q^{\ast} \left( (1+h^1)^{k} \cdot l_{2n+1-k} \right) \\

  & = \sum_{i\geq 0} {k \choose i} {\pi_X}_{\ast} \circ \pi_Q^{\ast}(l_{2n+1-k-i}) \\
 & \\
 & = \sum_{i\geq 0} {k \choose i} e_{k+i}
\end{array}\]
in $\text{Ch}(\bar{X})$ if $\text{dim}(h)$ is even, in $\text{Ch}_K(\bar{X})$ otherwise.
\end{proof}

Note that Theorem 7.2 is also a direct consequence of the quadratic case \cite[Theorem 4.1]{Grass},
see Remark 4.12.

\section{Canonical dimension}

In this section, we compute the canonical $2$-dimension $\text{cdim}_2(X)$ of the maximal unitary grassmannian $X$ associated with a non-degenerate $K/F$-hermitian form $h$
in terms of the $J$-invariant $J(h)$.

\medskip

We recall the definition of the canonical $2$-dimension of a variety (see \cite{cd} for an introduction on canonical dimension and for a more geometric definition). 

Let $X$ be an $F$-variety. An isotropy field $L$ of $X$ is an extension $L/F$ such that $X(L)\neq \emptyset$ (note that if $X$ is a maximal unitary grassmannian then this is the same thing as a splitting field of the corresponding hermitian form $h$).

An isotropy field $E$ is called \textit{$2$-generic} if for any isotropy
field $L$ there is an $F$-place $E \rightharpoonup L'$ for some finite extension $L'/L$ of odd degree
(see \cite[\S103]{EKM} for an introduction to $F$-places). For example, the function field $F(X)$ is $2$-generic (because it is generic).

The \textit{canonical $2$-dimension} $\text{cdim}_2(X)$ of $X$ is the minimum of the transcendence
degree $\text{tr.deg}_F(E)$ over all $2$-generic field extensions $E/F$.
If $X$ is smooth then $\text{cdim}_2(X)\leq \text{dim}(X)$. One says that $X$
is \textit{$2$-incompressible} if $\text{cdim}_2(X)=\text{dim}(X)$.

\medskip

The proof of the theorem below is a very slight modification of the corresponding one for quadratic forms,
see \cite[Theorem 90.3]{EKM}. Namely, at some point, one just has to consider the Chow rings $\text{Ch}_K$.
Nevertheless, we write down this modified version.

\medskip

The two following facts about $F$-places are used in the proof (contained in \cite[\S103]{EKM}).

One can compose $F$-places. In particular, any $F$-place $E \rightharpoonup L$ can be restricted
to a subfield $E'$ of $E$ containing $F$ (since field extensions over $F$ are $F$-places).

For any proper $F$-variety $Y$ equipped with an $F$-place $\pi: F(Y) \rightharpoonup L$, there is
a morphism $\text{Spec}(L)\rightarrow Y$. 

\medskip

The following $\text{Ch}_K$-version of the result \cite[Proposition 58.18]{EKM} for classical Chow groups will also be needed (see Remark 2.1).

\begin{prop}
\textit{Let $Z$ be a smooth $F$-scheme and $Y$ a $Z$-scheme.
Suppose there is a flat morphism of $F$-schemes
$f: Y\rightarrow Y'$.
If for every $y'\in Y'$, the pull-back homomorphism}
$\text{Ch}_K(Z)\rightarrow \text{Ch}_K\left( Y\times_{Y'} \text{Spec}\left(F(y')\right)\right)$
\textit{is surjective then
the homomorphism}
\[\begin{array}{ccl}
 \text{Ch}_K(Y')\otimes \text{Ch}_K(Z) & \rightarrow & \text{Ch}_K(Y) \\
\alpha \otimes \beta  & \mapsto & (f)_K^{\ast}(\alpha)\cdot \beta
\end{array}\]
\textit{is surjective.}
\end{prop}

\begin{thm}
\textit{Let $h$ be a non-degenerate $K/F$-hermitian form and $X$ the associated maximal unitary grassmannian.
Then }
\[\text{cdim}_2(X)=\text{dim}(X)-||J(h)||.\]

\end{thm}

\begin{proof}

Let $E$ be a $2$-generic isotropy field of $X$ with minimum transcendence degree $\text{cdim}_2(X)$
and $Y$ be the closure of the $F$-morphism $\text{Spec}(E)\rightarrow X$.
Since $F(Y)$ is a subfield of $E$, one has 
\begin{equation} \text{tr.deg}_F(E)\geq \text{dim}(Y). \end{equation}

Moreover, since $E$ is $2$-generic, there is an $F$-place
$E \rightharpoonup L$, with $L$ a field extension of $F(X)$ of odd degree.
Restricting $E \rightharpoonup L$ to $F(Y)$, one get an $F$-place $\pi: F(Y)\rightharpoonup L$, thus the existence of
a morphism $f: \text{Spec}(L)\rightarrow Y$. 
Let $g: \text{Spec}(L)\rightarrow X$ be the morphism induced by the extension $L/F(X)$
and let $Z$ be
the closure of the image of $(f,g):\text{Spec}(L)\rightarrow Y\times X$.
Then $[F(Z): F(X)]$ is odd  since it divides $[L: F(X)]$.
Thus, the image of $[Z]$ under the composition
\[\xymatrix{
\text{Ch}(Y\times X) \ar[r]^{\;\;\;\;(i\times 1)_{\ast}} &\text{Ch}(X^2)\ar[r]^{\;\;\;{p_2}_{\ast}} & \text{Ch}(X) },\]
where $i:Y\rightarrow X$ is the closed embedding and $p_2$ is the second projection, is equal to $[X]$.
It follows that $(i\times 1)_{\ast}([Z])_{F(X)}\neq 0$ in 
$\overline{\text{Ch}}_K(X^2)\subset \text{Ch}_K(\bar{X}^2)$.

We claim that the homomorphism
\begin{equation}\begin{array}{ccl}
 \text{Ch}_K(Y)\otimes \text{Ch}_K(X^2) & \rightarrow & \text{Ch}_K(Y\times X) \\
\alpha \otimes \beta  & \mapsto & (p_Y)^{\ast}_K(\alpha)\cdot (i\times 1)^{\ast}_K(\beta)
\end{array},\end{equation}
where $p_Y$ is the projection $Y\times X \rightarrow Y$, is surjective.
By Proposition 8.1, it suffices to show that for any $y\in Y$ the homomorphism 
$\text{Ch}_K(X^2)\rightarrow \text{Ch}_K(X_{F(y)})$ associated with the pull-back of the induced morphism 
$\text{Spec} (F(y))\times X\rightarrow X^{2}$ (where the second factor is the identity)
is surjective to prove the claim. 
The pull-back homomorphism $\text{Ch}(X^2)\rightarrow \text{Ch}(X_{F(y)})$ sends an 
element in the fiber
of the rational cycle $x_I\in \text{Ch}(\bar{X}^2)$ (introduced in (5.4)) under restriction to $e_I\in \text{Ch}(X_{F(y)})\simeq \text{Ch}(\bar{X})$. Therefore, since the classes $e_I$
generate $\text{Ch}_K(\bar{X})$ (Theorem 4.9), one deduces that the composition
\[\xymatrix{
\text{Ch}(X^2) \ar[r]&\text{Ch}(X_{F(y)})\ar[r] & \text{Ch}_K(X_{F(y)}) }\]
is surjective. Therefore the homomorphism $\text{Ch}_K(X^2)\rightarrow \text{Ch}_K(X_{F(y)})$ is also surjective and the claim is proven.

As a consequence of \cite[Proposition 49.20 and 58.17]{EKM}, one get that the diagram
\[\xymatrix{
 \text{Ch}_K(Y)\otimes \text{Ch}_K(X^2) \ar[r] \ar[d]_{{i_{\ast}}_K \otimes 1} &  \text{Ch}_K(Y\times X)  \ar[d]^{{(i\times 1)_{\ast}}_K} \\
 \text{Ch}_K(X)\otimes \text{Ch}_K(X^2) \ar[r] \ar[d] &  \text{Ch}_K(X^2)  \ar[d] \\
 \text{Ch}_K(\overline{X})\otimes \text{Ch}_K(X^2) \ar[r] &  \text{Ch}_K(\bar{X}^2)  
},\]
where the horizontal arrows are defined as in (8.4), is commutative.
Using the fact that 
$(i\times 1)_{\ast}([Z])_{F(X)}\neq 0$ in $\overline{\text{Ch}}_K(X^2)$
and the claim, one obtains that the composition
\[\xymatrix{
\text{Ch}_K(Y) \ar[r] &\text{Ch}_K(X) \ar[r] & \text{Ch}_K(\bar{X}) }\]
is non-trivial.
Therefore, by Proposition 6.1 one has $\text{dim}(Y)\geq \text{dim}(X)-||J(h)||$.
Thus, combining with inequality (8.3), one get
\[\text{cdim}_2(X)\geq \text{dim}(X)-||J(h)||.\]

By Proposition 6.1, there is a closed subvariety $Y\subset X$ of dimension
 $\text{dim}(X)-||J(h)||$ such that $[Y]\neq 0$ in $\overline{\text{Ch}}_K(X)$. In particular, one has 
$[Y]\neq 0$ in $\overline{\text{Ch}}(X)$.
From this moment on, the remaining of the proof of the second inequality is strictly the same as the one in the orthogonal case, see \cite[Theorem 90.3]{EKM}.
\end{proof}

We recall that if $\text{dim}(h)=2n+1$ then $\text{dim}(X)=n(n+2)=3+5+\cdots+(2n-1)+(2n+1)$ and
if $\text{dim}(h)=2n+2$ then $\text{dim}(X)=(n+1)^2=1+3+5+\cdots+(2n-1)+(2n+1)$.

\begin{cor} 
\textit{The variety $X$ is $2$-incompressible if and only if $J(h)$ is empty.}
\end{cor}

\begin{rem}
We call the \textit{generic hermitian form} of dimension $d$
(for the fixed separable extension $K/F$)
the diagonal $K(t_1,\dots,t_d)/F(t_1,\dots,t_d)$-hermitian form $<t_1,\dots,t_d>$, 
where 
$t_1,\dots,t_d$ are variables. 
Since any hermitian form can be diagonalized (see \cite[Theorem 6.3 of Chapter 7]{Sch}),
any $d$-dimensional
$K/F$-hermitian form is a specialization of the generic hermitian form of dimension $d$.
In the case of a generic hermitian form, 
the maximal unitary grassmannian is $2$-incompressible by \cite[Theorem 8.1]{UG}.
Therefore, by Corollary 8.5, the $J$-invariant of a generic hermitian form is empty.
\end{rem}

\section{Motivic decomposition}

In this section, we determine the complete motivic decomposition of the $\text{Ch}_K$-motive $M^K(X)$
(see \S 2) of the maximal unitary grassmannian $X$ associated with a non-degenerate $K/F$-hermitian form $h$
in terms of the $J$-invariant $J(h)$ (Theorem 9.4). 

\medskip

For $I \subset [1,\; 2n+1]_{od}$ or $[3,\; 2n+1]_{od}$ depending on whether $\text{dim}(h)$ is respectively equal to $2n+2$ or $2n+1$, let $\bar{I}$ denote the complementary set.
We set $J=J(h)$.
One always has $||J||+||\bar{J}||=\text{dim}(X)$.
We also use notation introduced in the previous sections.
By the very definition of the $J$-invariant and Proposition 5.1, for any $S \subset \bar{J}$ and $L,L' \subset J$, the cycle 
\[\theta_{S,L,L'}:=x_S\cdot (e_L\times e_{J\backslash L'}), \]
where $x_S$ is defined in (5.4), belongs to $\overline{\text{Ch}}_K(X^2)$.
Note that $\theta_{S,L,L'}$ can be rewritten as
\[\sum_{M\subset S}e_{M\sqcup L}\times e_{(S\backslash M) \sqcup (J\backslash L')},\]
where $\sqcup$ is the disjoint union of sets.  

We write $\theta_L$ for $\theta_{\bar{J},L,L}$ and $\theta_{L,L'}$ for $\theta_{\bar{J},L,L'}$. Let $\Delta_{\bar{X}}$ denote the class of the diagonal in $\text{Ch}_K(\bar{X}^2)$.

\begin{lemme}
\begin{enumerate}[(i)]
\item \textit{The set 
$\mathcal{B}=\{\theta_{S,L,L'} \;| \; S \subset \bar{J}\;; L,L' \subset J \}$ is a 
$\mathbb{Z}/2\mathbb{Z}$-basis of}
$\overline{\text{Ch}}_K(X^2).$
\item \textit{One has}
\[\theta_{S_2,L_2,L'_2}\circ \theta_{S_1,L_1,L_1'}=\delta_{L_2,L_1'}\cdot \delta_{S_1\cup S_2, \bar{J}}\cdot \theta_{S_1\cap S_2,L_1,L_2'}.\]
\item \textit{As correspondences of degree $0$, the elements of $\{\theta_L\;|\;L\subset J\}$
are pairwise orthogonal projectors such that $\sum_{L\subset J} \theta_L=\Delta_{\bar{X}}$.}
\end{enumerate}
\end{lemme}

\begin{proof}

The set $\mathcal{B}$ is a free family. Indeed, assume that $\sum_{S,L,L'}\alpha_{S,L,L'}\cdot \theta_{S,L,L'}=0$ for some $\alpha_{S,L,L'}$ in $\mathbb{Z}/2\mathbb{Z}$. 
Choose $L_0\subset J$.
By multiplying the latter equation by $e_{\bar{J} \sqcup (J\backslash L_0)}\times 1$, one get
\[\sum_{S,L,L'}\alpha_{S,L,L'}\cdot \left(\sum_{M\subset S}
e_{\bar{J} \sqcup (J\backslash L_0)}\cdot  e_{M\sqcup L}
\times e_{(S\backslash M) \sqcup (J\backslash L')}\right)=0.\]
By using (5.6), taking the image of the previous equation under the homomorphism
${({p_2}_{\ast})}_K$ associated with the push-forward of the second projection $\bar{X}\times \bar{X}\rightarrow \bar{X}$, one get 
\[\sum_{S, L'}\alpha_{S,L_0,L'} \cdot e_{S \sqcup (J\backslash L')}=0.\]
Hence, since the family $\{e_I\}$ is free (it is even a basis of $\text{Ch}_K(\bar{X})$, see 
Theorem 4.9), one obtains that $\alpha_{S,L_0,L'}=0$
for any $S\subset \bar{J}$ and any $L'\subset J$. Therefore, the family $\mathcal{B}$ is free. 

Moreover, since the $\text{Ch}_K$-motive $M^K(\bar{X})$ is split, by the analogue of \cite[Proposition 64.3]{EKM}
for $\text{Ch}_K$-motives, the map 
\[\text{Ch}_K(\bar{X})\otimes \text{Ch}_K(\bar{X})\rightarrow \text{Ch}_K(\bar{X}\times \bar{X})\]
(given by the external product) is an isomorphism.
Thus, by Theorem 4.9, any element of $\text{Ch}_K(\bar{X}\times \bar{X})$ can be written as 
$\sum_{I_1,I_2}\alpha_{I_1,I_2}\cdot (e_{I_1}\times e_{I_2})$ where the sum runs over 
all subsets $I_1,\, I_2\subset [1,\; 2n+1]_{od}$ or $[3,\; 2n+1]_{od}$ depending on whether $\text{dim}(h)$ is respectively equal to $2n+2$ or $2n+1$, and $\alpha_{I_1,I_2}\in \mathbb{Z}/2\mathbb{Z}$.
Let $a=\sum_{I_1,I_2}\alpha_{I_1,I_2}\cdot (e_{I_1}\times e_{I_2})\in \overline{\text{Ch}}_K(X^2)$.
Then, for any subset $I_2$, the element
\[{p_{1}}_{\ast}\left(x_{\bar{I_2}}\cdot a \right)=\sum\limits_{\substack{I_1, M\\ M\subset \bar{I_1}\cap \bar{I_2}}}\alpha_{I_1,I_2 \sqcup M}\cdot e_{I_1 \sqcup M}\]
is rational. Hence, by Theorem 5.7, for any subset $I\not\subset J$, one has
\[\sum_{ M\subset \bar{I_2}\cap I}\alpha_{I \backslash M,I_2 \sqcup M}=0.\]
Furthermore, the above relations are linearly independant because,
for any subset $I_2$ and $I\not\subset J$, the sum 
$r_{I_2,I}:=\sum_{ M\subset \bar{I_2}\cap I}\alpha_{I \backslash M,I_2 \sqcup M}$ contains $\alpha_{I,I_2}$ and for any $(I_2',I')\neq (I_2,I)$ such that $\alpha_{I,I_2}\in r_{I_2',I'}$,
one has $|I_2'|<|I_2|$.
It follows that
\[\text{rank}_{\mathbb{Z}/2\mathbb{Z}}\,\overline{\text{Ch}}_K(X^2)\leq 2^{r}\cdot 2^r - 2^r\cdot (2^r-2^{|J|})
=2^{r+|J|}=|\mathcal{B}|\]
(recall that $\text{rank}_{\mathbb{Z}/2\mathbb{Z}}\,\text{Ch}_K(\bar{X})=2^r$, with  $r=\lfloor \text{dim}(h)/2 \rfloor$).
Consequently, the family $\mathcal{B}$ is a 
$\mathbb{Z}/2\mathbb{Z}$-basis of
$\overline{\text{Ch}}_K(X^2)$.


For any subsets $S_1, S_2\subset \bar{J}$ and $L_1,L_1',L_2,L_2'\subset J$, the composition of correspondences $\theta_{S_2,L_2,L'_2}\circ \theta_{S_1,L_1,L_1'}$ is equal to
\[\sum_{M_1,M_2\subset \bar{J}}({p_{13}}_{\ast})_K\left(
e_{M_1\sqcup L_1}\times (e_{\left((S_1\backslash M_1)\sqcup (J\backslash L_1')\right)} \cdot e_{M_2\sqcup L_2})\times e_{(S_2\backslash M_2)\sqcup (J\backslash L_2')}\right). \]
Hence, the assertion (ii) follows from (5.6).

By (ii), the elements of $\{\theta_L\;|\;L\subset J\}$
are pairwise orthogonal projectors.
Furthermore, one has the following identity
\[\sum_{L\subset J} \theta_L= \sum_{I \subset [1,\; 2n+1]_{od}} e_I\times e_{ \bar{I} }. \]
Therefore, for any $I \subset [1,\; 2n+1]_{od}$ (or $[3,\; 2n+1]_{od}$ for odd-dimensional $h$), one has 
$\left(\sum_{L\subset J} \theta_L\right)^{\ast}(e_I)=e_I$. Since the elements $e_I$ generate 
the ring $\text{Ch}_K(\bar{X})$, one deduces that $\sum_{L\subset J} \theta_L=\Delta_{\bar{X}}$ and (iii) is proven.
\end{proof}

The next proposition is the $\text{Ch}_K$-version of the Rost Lemma (see \cite[Proposition 1]{RPF} or \cite[Theorem 67.1]{EKM}
and Remark 2.1).

\begin{prop}
\textit{Let $Y$ and $Z$ be smooth proper $F$-varieties.
If a correspondence} $\alpha \in \text{Ch}_K(Y\times Y)$ \textit{is such that} $\alpha \circ \text{Ch}_K(Y_{F(z)})=0$
\textit{for every $z\in Z$ then} 
\[\alpha^{\text{dim}(Z)+1} \circ \text{Ch}_K(Z\times Y)=0.\]
\end{prop}

For any smooth proper variety $Y$, by the very definition of the category of $\text{Ch}_K$-motives, 
one has
$\text{End}^j\left(M^K(Y)\right)=\text{Ch}_K^{\text{dim}(Y)+j}(Y\times Y)$ for any $j$,
with the composition of endomorphisms given by the composition of correspondences.
If $h$ is a non-degenerate $K/F$-hermitian form and $x$ is a point of the associated maximal unitary grassmannian $X$
then $h$ splits over $F(x)$,
so $\text{Ch}_K(X_{F(x)})=\text{Ch}_K(\bar{X})$.
Hence, Proposition 9.2 (applied with $Y=Z=X$) implies the statement below, which says that Rost Nilpotence holds for $X$ at the level of $K$-Chow rings. 

\begin{cor}
\textit{Let $X$ be the maximal unitary grassmannian associated with a non-degenerate $K/F$-hermitian form. The kernel of the restriction ring homomorphism}
\[\text{End}^{\ast}\left(M^K(X)\right)\rightarrow \text{End}^{\ast}\left(M^K(\bar{X})\right)\]
\textit{consists of nilpotent elements.}
\end{cor}

We are now able to prove the following $\text{Ch}_K$-motivic decomposition (in the spirit of \cite[Theorem 5.13]{J-inv}).

\begin{thm}
\textit{Let $h$ be a non-degenerate $K/F$-hermitian form and $X$ the associated maximal unitary grassmannian.
Then the} $\text{Ch}_K$\textit{-motive of $X$ decomposes as }
\[M^K(X)\simeq \bigoplus_{L\subset J(h)} \mathcal{R}(h)\{||L||\},\]
\textit{where  $\mathcal{R}(h)$ is an indecomposable motive.}

\textit{Moreover, over a splitting field of $h$, the} $\text{Ch}_K$\textit{-motive $\mathcal{R}(h)$ decomposes as a sum of shifts of the Tate motive. More precisely, one has}
\[\overline{\mathcal{R}(h)}\simeq \bigoplus_{M\subset \overline{J(h)}} \mathbb{Z}/2\mathbb{Z} \{||M||\}. \]
\end{thm}

\begin{proof}

By Lemma 9.1(ii), one has
\[\theta_{L,N}\circ \theta_{L',N'}=\delta_{L,N'}\cdot\theta_{L',N}.\]
In particular, this implies that for any $L,L'\subset J$, the projectors $\theta_L$ and $\theta_{L'}$ are
isomorphic (in the sense of \cite[\S 2.1]{J-inv}), the isomorphism being given by $\theta_{L,L'}$ and 
$\theta_{L',L}$.

We claim that the projectors $\theta_L$ are indecomposable.
For any integers $0\leq k \leq ||\bar{J}||$ and $0\leq l\leq \text{dim}(X)$, we write
\[\overline{\text{Ch}}_K^{k,l}(X^2)\]
for the subspace of $\overline{\text{Ch}}_K(X^2)$ spanned by the elements
$\theta_{S,L,L'}$ with $||S||\leq k$ and codimension $\leq l$.
For any correspondences $\alpha, \alpha'\in \text{Ch}_K(\bar{X}^2)$ one has
\[\text{codim}(\alpha\circ \alpha')=\text{codim}(\alpha)+\text{codim}(\alpha')-\text{dim}(X).\]
Thus, by Lemma 9.1(i) (or Lemma 9.1(ii)),
the space $\overline{\text{Ch}}_K^{||\bar{J}||,\text{dim}(X)}(X^2)$ equipped with the composition of correspondences is a ring.
Furthermore, using the same codimensional considerations and the formula of Lemma\,9.1(ii), one sees that, 
for any $0\leq k \leq ||\bar{J}||$ and $0\leq l\leq \text{dim}(X)$, the subspace 
$\overline{\text{Ch}}_K^{k,l}(X^2)$
is an ideal of $\overline{\text{Ch}}_K^{||\bar{J}||,\text{dim}(X)}(X^2)$.
We denote by $\mathcal{A}$ the factor ring of 
$\overline{\text{Ch}}_K^{||\bar{J}||,\text{dim}(X)}(X^2)$ by the sum
of ideals 
$\overline{\text{Ch}}_K^{k,l}(X^2)$ for all $(k,l)\in \left([0,\; ||\bar{J}||]\times[0,\; \text{dim}(X)]\right)\backslash \{(||\bar{J}||,\text{dim}(X))\}$.
By Lemma\,9.1(i), a $\mathbb{Z}/2\mathbb{Z}$-basis of $\mathcal{A}$ is given by the classes of
the elements $\theta_{L,L'}$ with $||L||=||L'||$.
Hence, since $\theta_{L,L'}^{\ast}(e_N)=\delta_{N,L'}\cdot e_L$ for any $N\subset J$ (by (5.6)), it follows from
Lemma 9.1(ii) that the assignment $\theta_{L,L'}\mapsto \theta_{L,L'}^{\ast}$ for all
$\theta_{L,L'}$ with $||L||=||L'||$ gives rise to an anti-isomorphism between $\mathcal{A}$ 
and the product of matrix rings 
\[\prod_{ k=0}^{\text{dim}(X)}\text{End}\left( \mathcal{E}_k\right),\]
where $\mathcal{E}_k$ is the subspace of $\overline{\text{Ch}}_K(X)$ spanned by the elements
$e_N$ with $||N||=k$.
Under this identification, for any $L\subset J$, the class of $\theta_L$ in $\mathcal{A}$ corresponds to an idempotent element 
of rank $1$ and therefore is indecomposable.
Moreover, by Lemma 9.1(ii) and codimensional reasons
(for any correspondence $\alpha\in \text{Ch}_K(\bar{X}^2)$ of codimension $<\text{dim}(X)$, one has
$\text{codim}\left(\alpha^{\circ i}\right)>\text{codim}\left(\alpha^{\circ (i+1)}\right)$), the kernel of the projection
\[\pi: \overline{\text{Ch}}_K^{||\bar{J}||,\text{dim}(X)}(X^2)\rightarrow \mathcal{A}\]
is nilpotent. Thus, the projectors $\theta_L$ are indecomposable in $\overline{\text{Ch}}_K(X^2)$. The claim is proven.


Consequently, combining with Lemma 9.1(iii) and Rost Nilpotence Corollary 9.3, one get that there exists a family 
$\{\psi_L\;|\;L\subset J \}$ of
pairwise orthogonal projectors in 
$\text{Ch}_K^{\text{dim}(X)}(X\times X)$, satisfying $\overline{\psi_L}=\theta_L$,
all isomorphic to $\psi_J$ (with respect to the correspondences $\theta_{L, J}$), and such that $\sum_{L\subset J}\psi_L=\Delta_X$
(see \cite[Proposition 2.6]{J-inv}). Note that Rost Nilpotence implies that the projectors $\psi_L$
are also indecomposable. Let $\mathcal{R}(h)$ denote the indecomposable $\text{Ch}_K$-motive $(X, \psi_\emptyset)^K$. Thus, for any $L\subset J$, one has
$(X, \psi_L)^K=\mathcal{R}(h)\{||L||\}$ (since $\text{codim}(\theta_{L, J})=\text{dim}(X)-||J\backslash L||$).
In other words, one has the desired $\text{Ch}_K$-motivic decomposition of $M^K(X)$.

We prove now the last assertion of the theorem. Over a splitting field, one has 
$\overline{\mathcal{R}(h)}=(\bar{X}, \theta_{\emptyset})^K$. Moreover the writing 
$\theta_{\emptyset}=\sum_{M\subset \bar{J}}e_M\times e_{(\bar{J}\backslash M)\sqcup J}$ is a decomposition as a sum of pairwise orthogonal projectors in 
$\text{Ch}_K(\bar{X}\times \bar{X})$. These projectors are all isomorphic to $1\times e_{\bar{J} \sqcup J}$ 
with $(\bar{X},e_M \times e_{(\bar{J}\backslash M)\sqcup J})^K\simeq (\bar{X}, 1\times e_{\bar{J}\sqcup J})^K\{||M||\}$ (recall that $e_{\bar{J}\sqcup J}$ is the class of a rational point).
Furthermore, one easily checks that the $\text{Ch}_K$-motive $(\bar{X}, 1\times e_{\bar{J}\sqcup J})^K$ is isomorphic the the Tate motive $\mathbb{Z}/2\mathbb{Z}$ in the category of $\text{Ch}_K$-motives.

The theorem is proven.
\end{proof}

The following statement can be viewed as the $\text{Ch}_K$-version of the Krull-Schmidt principle
for the $F$-variety $X$.

\begin{cor}
\textit{Any direct summand of $M^K(X)$ is isomorphic to a sum of shifts of the motive $\mathcal{R}(h)$.
}
\end{cor}

\begin{proof}
We use notation and material introduced in the proof of Theorem 9.4.
Any idempotent in the ring $\mathcal{A}$ is isomorphic (in the sense of \cite[\S 2.1]{J-inv}) 
to a sum of classes of $\theta_L$ for some $L\subset J$.
Moreover, since $\text{Ker}\left(\pi \right)$ is nilpotent, the projection $\pi$ lifts isomorphisms,
see \cite[Proposition 2.6]{J-inv}.
\end{proof}

\begin{rem}
It follows from Theorem 9.4 that $M^K(X)$ is indecomposable if and only if $J(h)$ is empty, i.e., if
and only if $X$ is $2$-incompressible.
In particular, this applies to generic hermitian forms.
\end{rem}

\begin{rem}

It follows from (6.6) that there is a group isomorphism
\[\text{Ch}(M_e \otimes M_e) \rightarrow \text{Ch}_K(X\times X),\]
with $M_e$ the essential Chow motive of $X$.
Moreover, let $\varphi\in \text{Ch}(X\times X)$ be the projector giving $M_e$.
Then the group $\text{Ch}(M_e \otimes M_e)=\varphi\circ \text{Ch}(X\times X)\circ \varphi$
is equipped with the ring structure given by the composition of correspondences
(the neutral element being $\varphi$).
Since the above isomorphism is the restriction to $\varphi\circ \text{Ch}(X\times X)\circ \varphi$ of the projection $ \text{Ch}(X\times X) \rightarrow \text{Ch}_K(X\times X)$, it is a ring isomorphism
(with $ \text{Ch}_K(X\times X)$ equipped with the composition of $\text{Ch}_K$-correspondences).
Hence, 
the essential motive $M_e$
is decomposable if and only if the
$\text{Ch}_K$-motive $M^K(X)$ is decomposable.
\end{rem}

\section{Comparison with quadratic forms}

Let $h$ be a non-degenerate $K/F$-hermitian form and $q$ the associated non-degenerate $F$-quadratic form.
In this section, we compare the $J$-invariant $J(h)$ with the $J$-invariant $J(q)$ as defined by A.\;Vishik
in \cite{Grass}. We recall that if $\text{dim}(h)=2n+2$ then $J(q)$ is a subset of $[0,\; 2n+1]$, otherwise 
-- if $\text{dim}(h)=2n+1$ then it is a subset of $[0,\; 2n]$ and in both cases $J(h)$ is a 
subset of $[1,\; 2n+1]_{\text{od}}$.

Let $X$ be the maximal hermitian grassmannian of $h$ and
$Y$ the maximal orthogonal grassmannian of $q$.
For even-dimensional $h$, let $im$ denote the
natural closed embedding $X\hookrightarrow Y$ (see Remark 4.12).

\begin{prop}
\textit{One has}
\[J(q)=\left \{\begin{array}{cl} 
 J(h) \cup [0,\; 2n]_{\text{ev}}& \;\;\textit{if} \;\; \text{dim}(h)=2n+2 ; \\ 
 & \\
 \left[1,\; 2n \right] & \;\; \textit{if} \;\; \text{dim}(h)=2n+1 , 
\end{array}\right.\]
\textit{where} $\left[0,\; 2n \right]_{\text{ev}}$ \textit{stands for the even part of the set $[0,\; 2n]$.}
\end{prop}

\begin{proof}
Assume that $\text{dim}(h)=2n+2$.
We use notation and material introduced in Remark 4.12.
Since $\textit{im}^{\ast}(z_k)=e_k$ in $\text{Ch}(\bar{X})$
for any $1\leq k\leq 2n+1$ (see Remark 4.12) and 
the generators $z_k$ of the ring $\text{Ch}(\bar{Y})$ define $J(q)$ the same way
the elements $e_k\in \text{Ch}_K(\bar{X})$, for $k$ odd, define $J(h)$ (see \cite{Grass} or \cite{EKM}),
one has $J(q)_{\text{od}}\subset J(h)$.
Moreover, since the quadratic form $q$ is obtained from a $K/F$-hermitian form, the absolute Witt indices of $q$ are even.
Hence, by \cite[Proposition 88.8]{EKM} (this is the quadratic equivalent of Proposition 6.12, with the $J$-invariant defined in the opposite way), the set $[0,\; 2n]_{\text{ev}}$ is contained in $J(q)$.
Furthermore, by \cite[Corollary 9.3]{UG}, the varieties $X$ and $Y$ have the same canonical $2$-dimension.
Thus, by \cite[Theorem 90.3]{EKM} and Theorem 8.2, one has
\[||J(q)||-||J(h)||=\text{dim}(Y)-\text{dim}(X)=0+2+\cdots +(2n-2)+2n.\]
Consequently, one has $J(q)=J(h) \cup [0,\; 2n]_{\text{ev}}$.

Assume that $\text{dim}(h)=2n+1$. In this case, one has $\text{cdim}_2(Y)=0$ (see \cite[Corollary 9.3]{UG}). Therefore, by \cite[Theorem 90.3]{EKM}, one has 
\[ ||J(q)||=\text{dim}(Y)=\frac{2n(2n+1)}{2}.\]
Moreover $0\notin J(q)$ because the discriminant of $q$ is not trivial.
Consequently, one has $J(q)=[1,\; 2n]$.
\end{proof}

\begin{rem}
Proposition 10.1 allows one to recover the smallest value of the $J$-invariant of a non-degenerate quadratic form $q$ associated with a hermitian form $h$ of even dimension $2n+2$ over a quadratic separable field extension of the base field, i.e., $q$ is given by the tensor product of a $(2n+2)$-dimensional bilinear form by a binary quadratic form.
Namely, this value is $[0,\; 2n]_{\text{ev}}$ and it is obtained for $h$ generic.
This was originally proven by N.\,A.\,Karpenko, see \cite[Corollary 9.4]{UG}.
\end{rem}

\begin{rem}
For even-dimensional non-degenerate $K/F$-hermitian forms, Proposition 10.1 and its proof, combined with Theorem 5.7, provide another argument for the surjectivity
of the homomorphism
\[\overline{\text{Ch}}(Y)\rightarrow \overline{\text{Ch}}_K(X),\]
associated with the pull-back $\textit{im}^{\ast}$. 
This was originally observed by Maksim Zhykhovich, see
\cite[Lemma 9.8]{UG}. 
\end{rem}

\bibliographystyle{acm} 
\bibliography{references}
\end{document}